\documentclass[11pt,reqno]{amsart}  
\usepackage{amssymb}
\usepackage{color}
\usepackage{enumerate}
\allowdisplaybreaks[1]
\newtheorem{thm}{Theorem}[section]
\newtheorem{cor}[thm]{Corollary}
\newtheorem{dfn}[thm]{Definition}
\newtheorem{prop}[thm]{Proposition}
\newtheorem{lem}[thm]{Lemma}
\newtheorem{rem}[thm]{Remark}
\renewcommand{\qed}{\qquad\kern1pt 
 \vbox{\hrule height 0.6pt      
  \hbox{\vrule width 0.6pt 
   \vbox{\vskip 6pt}  
   \hskip 3pt
  \vrule width 1.3pt} 
 \hrule depth 1.3pt}     
\kern1pt}
\newcommand{\sect}[1]{\section{#1} \setcounter{equation}{0}}

\newcommand{\pt}{\partial}           
\newcommand{\R}{\mathbb R}

\newcommand{\Z}{\mathbb Z}
\newcommand{\N}{\nabla}

\newcommand{\gm}{\gamma}
\newcommand{\lam}{\lambda}

\newcommand{\dl}{\delta}
\newcommand{\Del}{\Delta}
\newcommand{\s}{\sigma}
\renewcommand{\r}{\rho}
\renewcommand{\t}{\tau}

\newcommand{\hr}{\hookrightarrow}
\newcommand{\wt}{\widetilde }
\newcommand{\wh}{\widehat }
\renewcommand{\div}{{\rm{div\,}}}

\newcommand{\dB}{\dot{B}}

\newcommand{\fB}{\widehat{\dot{B}}}
\newcommand{\cfB}{\widehat{\dot{B}}_{p,1}}
\newcommand{\ve}{\varepsilon}

\newcommand{\mG}{\mathcal{G}}

\newcommand{\al}{\alpha}
\newcommand{\supp}{\text{supp\;}}
\renewcommand{\L}{\mathcal{L}}

\newcommand{\les}{\lesssim\;}
\newcommand{\dsp}{\displaystyle}
\newcommand{\F}{\mathcal{F}}
\newcommand{\mN}{\mathcal{N}}
\newcommand{\Id}{\text{Id}}
\newcommand{\B}{\mathcal{B}}
\newcommand{\Dj}{\dot{\Delta}_j}

\begin{document}
\title{Analyticity and asymptotic behavior of solutions to the compressible Navier-Stokes-Korteweg equations 
with the zero sound speed in scaling critical spaces} 
\maketitle

\maketitle
\vskip5mm
{\small
\hskip7mm\noindent
{\normalsize \sf Takayuki Kobayashi}
\vskip1mm\hskip1.5cm
        Graduate School of Engineering Science, The University of Osaka \hfill\break\indent\hskip1.5cm
        Osaka 560-8531, Japan \hfill\break\indent\hskip1.5cm 
        kobayashi@sigmath.es.osaka-u.ac.jp}
\vskip5mm
{\small
\hskip7mm\noindent
{\normalsize \sf Ryosuke Nakasato}
\vskip1mm\hskip1.5cm
        Faculty of Engineering, Shinshu University \hfill\break\indent\hskip1.5cm
        Nagano 380-8553, Japan \hfill\break\indent\hskip1.5cm
        nakasato@shinshu-u.ac.jp}

\vspace*{2mm}
\begin{abstract} 
We consider the initial-value problem in the $d$-dimensional Euclidean space $\mathbb{R}^d$ $(d \ge 3)$ for the compressible 
Navier-Stokes-Korteweg equations under the zero sound speed case 
(namely, $P'(\rho_*)=0$, where $P=P(\rho)$ stands for the pressure). 
The system is well-known as the Diffuse Interface model describing the motion of 
a vaper-liquid mixture in a compressible viscous fluid. 
The purposes of this paper are to obtain the global-in-time solution around the constant equilibrium states 
$(\rho_*,0)$ $(\rho_*>0)$ satisfying the 
estimate on the analyticity as established by Foias-Temam (1989),  and investigate the $L^p$-$L^1$ type 
time-decay estimates in scaling critical settings based on Fourier-Herz spaces. 
In addition, we also derive the first order asymptotic formula with higher derivatives 
for solutions as the application of the analyticity. 
\end{abstract}


\baselineskip=5.6mm
\sect{Introduction}

\subsection{Compressible Navier-Stokes-Korteweg system}
This paper investigates the analiticity and asymptotic behavior of global solutions   
to the initial value problem of the following compressible   
Navier-Stokes-Korteweg system in the $d$-dimensional Euclidean space $\R^d$ ($d \ge 3$): 
\vskip-1.5mm
\begin{equation} \label{eqn;NSK} 
\left\{
\begin{split}
& \pt_t\r+\div (\r u)=0, 
&t>0, x \in \R^d, \\
&\pt_t (\r u) 
+\text{div}\left(\r u \otimes u \right)+\N P(\r)= \text{div} \left(\mathcal{T}(\r,u)+\mathcal{K}(\r)\right)
&t>0, x \in \R^d, \\
&(\r, u)|_{t=0}=(\r_0,u_0), &x \in \R^d,  
\end{split}
\right.
\end{equation}
\vskip1.5mm
\noindent
where $\r=\r(t,x):\R_+ \times \R^d \to \R_+$ and  
$u=u(t,x):\R_+ \times \R^d \to \R^d$ denote the density of the fluid and the velocity of the fluid, respectively. 
The pressure $P=P(\r)$ is assumed to be a {\it real analytic function} of $\r$ in a neighborhood of 
$\rho_*>0$.  
In addition, we assume that the pressure $P$ satisfies $P'(\rho_*)=0$, namely the sound speed 
$\gm:=\sqrt{P'(\rho_*)}$ is equal to $0$.  
The viscous stress tensor $\mathcal{T}(\r,u)$ with the viscosity coefficients $\lam=\lam(\r)$, $\mu=\mu(\r)$ is given by 
$$
\mathcal{T}(\r,u)=2 \mu(\r) D(u)+\lam(\r) \div u \,\text{Id}. 
$$ 
Here $\text{Id}$ denotes the identity matrix which is given by $\Id=(\dl_{ij})_{ij}$ 
with $\dl_{ij}$ designating the Kronecker delta and the deformation tensor $D(u)$ is defined by  
$$
D(u)=\frac{1}{2}(\N u+{}^t(\N u)) \;\;\text{with}\;\;(\N u)_{ij}=\pt_i u_j.  
$$
Throughout this paper,  we assume that $\mu$, $\lam$ are given constants satisfying the standard elliptically conditions $\mu>0$ and $\nu:=\lam+2\mu>0$. 
For  $\kappa=\kappa(\r)$, 
the Korteweg stress tensor $\mathcal{K}(\r)$ in the second equation of \eqref{eqn;NSK} is given by 
\begin{equation} \label{eqn;Korteweg}
\mathcal{K}(\r)=\frac{\kappa}{2}(\Delta \r^2-|\N \r|^2) \,\text{Id} -\kappa \N \r \otimes \N \r, 
\end{equation}
where $\N \r \otimes \N \r$ stands for the tensor product $(\pt_i \r \pt_j \r)_{ij}$.  In general, it is natural that $\kappa$ depends on $\r$. 
For simplicity, in this paper, we assume that $\kappa$ is a positive constant. 

In this paper, we focus on a solution for the density that is close to a constant equilibrium state  
$\rho_*>0$ at spatial infinity. Setting $m:=\r u$,  let us reformulate the problem \eqref{eqn;NSK} as followings:  
\begin{equation} \label{eqn;mNSK} 
\left\{
\begin{split}
& \pt_t\r+\div m=0, 
&t>0, x \in \R^d, \\
&\pt_t m 
+\text{div}\left(\r^{-1} m \otimes m \right)+\N P(\r)= \mathcal{L}\left(\r^{-1}m \right)+\text{div} \left(\mathcal{K}(\r)\right)
&t>0, x \in \R^d, \\
&(\r, m)|_{t=0}=(\r_0,m_0), &x \in \R^d, 
\end{split}
\right.
\end{equation}
where $\mathcal{L}$ is the Lam\'{e} operator defined by $\mu \Delta + (\lam+\mu) \N \div$. 
In what follows, we perform our analysis in the above momentum representation \eqref{eqn;mNSK}. 
The merit of the momentum representation is to maintain the divergence form, which the standard velocity formulation does not equip. 
Taking into account the nonlinear terms equipped with the divergence form, 
we may be possible to obtain better results than the previous studies on the time-decay and analyze the asymptotic behavior of solutions with the critical regularity.   

\subsection{Known Results} 
The compressible Navier-Stokes-Korteweg system \eqref{eqn;NSK} is well-known as the Diffuse Interface (DI) model describes the motion of 
a vaper-liquid mixture in a compressible viscous fluid. 
The primitive theory regarding the DI model was first proposed by Van der Waals \cite{Va}. Later, Korteweg \cite{Ko} introduced the stress tensor 
including the term $\N \r \otimes \N \r$. 
The rigorous derivation of the corresponding equations \eqref{eqn;NSK} was given by Dunn-Serrin \cite{Du-Se}. 

First of all, we would like to mention about the results on the case $\gm>0$.  
For the existence of a strong solution to the problem \eqref{eqn;NSK}, Hattori-Li \cite{Ha-Li,Ha-Li2} established in a inhomogeneous Sobolev framework.   
The one of the purpose of this paper is to obtain a global-in-time solution of the problem \eqref{eqn;NSK}  
in a scaling {\it critical} $L^p(\R^d)$ framework. In order to explain what we mean by critical regularity, 
we focus on a nearly scaling invariance property for the compressible Navier-Stokes-Korteweg system. 
For a solution $(\r,u)$ to the problem \eqref{eqn;NSK},  
the scaled functions $(\r_\al,u_\al)$ with the parameters $\al>0$ given by 
\begin{equation} \label{eqn;c1}
 \r_\al(t,x)= \r(\al^2t,\al x), \quad 
 u_\al(t,x)= \al u(\al^2 t, \al x),  
\end{equation}
also satisfy the same problem without the initial conditions provided that the pressure $P$ is changed into $\al^2 P$.   
By virtue of Fujita-Kato's principle \cite{Fj-Kt}, one can find that 
the invariant function class under the scaling \eqref{eqn;c1} is given by 
{\it critical function spaces}.  
The critical space for the initial data $(\r_0,u_0)$ in problem \eqref{eqn;NSK}, 
for instance, is given by the homogeneous Besov space   
$\dB^{d/p}_{p,\s}(\R^d) \times \dB^{-1+d/p}_{p,\s}(\R^d)$ 
(for the definition of Besov spaces, see Definition 2.15 in \cite{B-C-D}). 
Focusing on the above nearly scaling invariances, Danchin-Desjardins \cite{Da-De} established the result on the small data global existence 
of the problem \eqref{eqn;NSK} in the critical Besov space based on $L^2(\R^d)$. 
Recently, Charve-Danchin-Xu \cite{Ch-Da-Xu} obtained the global well-posedness and Gevrey analyticity for the global solution in more 
general critical $L^p(\R^d)$ framework. 

Regarding the asymptotic stability for the global solution to the problem \eqref{eqn;NSK} around the constant equilibrium state $(\r_*,0)$, 
there are numerous studies up to the present. 
In the case of the smooth initial data, the many authors have established the $L^2$-$L^1$ type time decay estimate, 
which is the same decay rate as the $L^2(\R^d)$-norm of the fundamental solution to the heat equation provided the initial data belongs to $L^1(\R^d)$ (cf. \cite{Ko-Ts,Ta-Wa-Xu,Ta-Zh,Wa-Ta,Zh-Ta}). 
In the scaling critical case, Chikami-Kobayashi \cite{Ch-Ko} showed that the $\dot{B}_{2,1}^s$-$\dot{B}_{2,\infty}^{-d/2}$ type time-decay estimate holds true 
for the global solution obtained by Danchin-Desjardins \cite{Da-De}. 
Later, Kawashima-Shibata-Xu \cite{Ka-Sh-Xu} extended the Chikami-Kobayashi's decay result to the critical $L^p(\R^d)$-Besov framework. 
Recently, in \cite{Ko-Na}, we establish the estimate on the linear approximation in critical Fourier-Besov spaces (see Definition \ref{dfn;FB}) and the decay estimate 
with diffusion wave in critical $L^2$-Besov spaces. 
In regard to the result on the low mach number limit for the 2D system in critical $\wh{L}^p$ framework, 
we refer to the recent work by Fujii-Li \cite{Fu-Li}. 

On the other hand, it is physically important to consider the system \eqref{eqn;NSK} in the case $\gm=0$ or $\gm <0$ because the system \eqref{eqn;NSK} was deduced by 
using Wan der Waals potential which contains the non-monotone pressure. For details, see for instance \cite{Ko-Ts}. 
In this paper, we focus on the case of $\gm=0$. In \cite{Da-De}, they constructed the local-in-time solution of \eqref{eqn;NSK} with $\gm=0$ in critical $L^2$-Besov spaces. 
Chikami-Kobayashi \cite{Ch-Ko} globally extended the local solution obtained by Danchin-Desjardins \cite{Da-De} under an additional low-frequency assumption 
and they also established the 
$\dot{B}_{2,1}^s$-$\dot{B}_{2,\infty}^{-d/2}$ type time-decay estimate. 
Later on, Zhang \cite{Zh} obtained global solution in generalized critical $L^2$-$L^p$ framework with $d \ge 4$.  
In the inhomogeneous Sobolev setting,  
Kobayashi-Murata \cite{Ko-Mu} and Kobayashi-Tsuda \cite{Ko-Ts} investigated the existence of global solution and time-decay of solutions. 
Recently, Song-Xu \cite{So-Xu, So-Xu2} established the Gevrey analyticity and time-decay estimate of global solutions to the problem 
\eqref{eqn;NSK} with $\nu^2 \ge 4 \kappa$ in critical $L^p$-Besov spaces.  

\subsection{Function spaces} 
Before stating the main statements of this paper, we introduce some 
notation and definitions. 
For $d \ge 1$ and $1\le p\le\infty$, let $L^p = L^p(\R^d)$ be the Lebesgue space. 
For any $f$ belonging to the {\it Schwartz class} $\mathcal{S}=\mathcal{S}(\R^d)$, 
the Fourier transform of $f$ denoted by $\wh{f}=\wh{f}(\xi)$ or $\mathcal{F}[f]=\mathcal{F}[f](\xi)$ is  
$$
 \mathcal{F}[f](\xi)(=\wh{f}(\xi)):= 
   \frac{1}{(2 \pi)^{d/2}}\int_{\R^d} e^{-i x \cdot \xi}f(x) \,dx.             
$$
Similarly, 
for any $g=g(\xi)$ belonging to $\mathcal{S}(\R^d_\xi)$, the Fourier {\it inverse} transform $\mathcal{F}^{-1}[g]=\mathcal{F}^{-1}[g](x)$ is then defined as 
$$
 \mathcal{F}^{-1}[g](x) := \frac{1}{(2 \pi)^{d/2}} \int_{\R^d} e^{i x \cdot \xi} g(\xi) \,d\xi.             
$$
Let $\{ \phi_j \}_{j \in \Z}$ be the Littlewood-Paley 
dyadic decomposition of unity, i.e., for a non-negative radially 
symmetric function $\phi \in \mathcal{S}$, we set (for the construction 
of $\{ \phi_j \}_{j \in \Z}$, see e.g., \cite{B-C-D}, \cite{S}) 
\begin{equation*}
\wh{\phi_j}(\xi) := \wh{\phi}(2^{-j}\xi) \;(j \in \Z), \;\;
\sum_{j \in \Z}\wh{\phi}_j(\xi) = 1 \; (\xi \neq 0) \;\;
\text{and} \;\;
\supp \wh{\phi}
\subset
\{\xi \in \R^{d};\frac{1}{2} \le |\xi| \le 2\}. 
\end{equation*}
\begin{dfn}[{\it Homogeneous Fourier-Besov spaces}\hspace{0mm}] Let $d \ge 1$, \label{dfn;FB}
$s \in \R$, $1 \le p,\s \le \infty$ 
and $\mathcal{S}'=\mathcal{S}'(\R^d)$ be the space of tempered distributions. 
We define the homogeneous Fourier-Besov space $\fB{_{p,\s}^s} = \fB{_{p,\s}^s}(\R^d)$ as follows:
\begin{align*}
  \fB{_{p,\s}^s}(\R^d)
  :=\{f\in\mathcal{S}';\widehat{f}\in L^1_{loc}(\R^d), \,\|f\|_{\fB{_{p,\s}^s}} < \infty\}, \quad
  \|f\|_{\fB{_{p,\s}^s}}:=\left\|\left\{2^{sj}\|\Dj f\|_{\wh{L}^p}\right\}_{j\in\Z}\right\|_{\ell^\s},
\end{align*} 
where $\Dj f:=\mathcal{F}^{-1}[\wh{\phi}_j \wh{f}\,]$ for some $f \in \mathcal{S}'$ and 
$\|f\|_{\wh{L}^p}:=\|\wh{f}\|_{L^{p'}_\xi}$.  
\end{dfn}

\vskip2mm
Taking into account the time variable, we give the definition of the space-time mixed spaces 
introduced by Chemin-Lerner \cite{Ch-Le}.
\vskip2mm
\begin{dfn}[{\it Chemin-Lerner spaces}\hspace{0mm}] 
Let $d \ge 1$, $s \in \R$, $1 \le p \le \infty$, $1 \le \s,r \le \infty$ and 
$I=[0,T]$ with $T \in \R_+$.  
We define the Chemin-Lerner space based on the Fourier-Besov space as follows:
\begin{align*}
  \wt{L^r(I;}\fB{_{p,\s}^s})
  :=\overline{C(I;\mathcal{S}_0)}^{\|\cdot\|_{\wt{L^r(I;}\fB{_{p,\s}^s})}}, \quad
  \|f\|_{\wt{L^r(I;}\fB{_{p,\s}^s})}:=\left\|\left\{2^{sj}\|\Dj f\|_{L^r(I;\wh{L}^{p})}\right\}_{j\in\Z}\right\|_{\ell^\s}, 
\end{align*}
where $\mathcal{S}_0=\mathcal{S}_0(\R^d)$ is the set of functions in the Schwartz class 
$\mathcal{S}(\R^d)$  
whose Fourier transform are supported away from $0$. 
\end{dfn}
\begin{rem} 
By Minkowski's inequality, we note that the following continuous embeddings hold between the 
Chemin-Lerner space $\wt{L^r(I;}\fB{_{p,\s}^s})$ and the Bochner space 
$L^r(I;\fB{_{p,\s}^s})$:
\begin{align*}
      \wt{L^r(I;}\fB{_{p,\s}^s}) 
       \hookrightarrow L^r(I;\fB{_{p,\s}^s}) \quad \text{if}\;\; r \ge \s, \quad
      L^r(I;\fB{_{p,\s}^s}) 
       \hookrightarrow \wt{L^r(I;}\fB{_{p,\s}^s}) \quad \text{if}\;\; r \le \s 
\end{align*}
(for the definition of Bochner spaces, see the end of this section). 
\end{rem}

\begin{dfn}[{\it The space $CL_{T}^{(p,\s)}$}] \label{dfn;CL_T}
For $T>0$ and $1 \le p, \s \le \infty$, we denote by $CL_{T}^{(p,\s)}$ the space of functions 
$(f,g)$ 
such that 
\begin{align*}
\|(f,g)\|_{CL_{T}^{(p,\s)}}:=\|(|\N|f,\,&g)\|_{\wt{L^\infty(I};\fB{_{p,\s}^{-3+\frac{d}{p}}})
                \cap \wt{\;L^1(I;}\fB{_{p,\s}^{-1+\frac{d}{p}}})} \\
&+\|(|\N|f,g)\|_{\wt{\;L^\infty(I};\fB{_{p,1}^{-1+\frac{d}{p}}})
                \cap L^1(I;\fB{_{p,1}^{1+\frac{d}{p}}})} < \infty. 
\end{align*}
\end{dfn}

\section{Main results} 

\subsection{Main result I: Analyticity and time-decay of solutions} 
In what follows, we set $a:=\r-\r_*$ for the constant equilibrium state $\r_*>0$. We rapidly obtain the following equations for the 
perturbation denoted by $(a, m)$:  
\begin{equation} \label{eqn;mNSK3} 
\left\{
\begin{split}
& \pt_ta+\div m=0, 
&t>0, x \in \R^d, \\
&\pt_t m-\frac{1}{\r_*}\mathcal{L}m-\kappa \N \Delta a = \mN(a,m)
&t>0, x \in \R^d, \\
&(a, m)|_{t=0}=(a_0,m_0), &x \in \R^d,   
\end{split}
\right.
\end{equation}
where 
$a_0=\r_0-\r_*$ and 
the nonlinear part $\mN(a,m)$ is given by 
\begin{align*} 
\mN(a,m)=-\frac{1}{\r_*} \div (m \otimes m)&+ \text{div} \left(\frac{a}{\r_*(a+\r_*)} m \otimes m\right) \\
&-P'(a+\r_*) \N a -\mathcal{L}\left( \frac{a}{\r_*(a+\r_*)} m \right)+ \div \mathcal{K}(a). 
\end{align*}

We shall state the main results of this paper for \eqref{eqn;mNSK3}. 

\begin{thm}[{\it Global existence and Analyticity of solution}\hspace{0mm}] \label{thm;GWP_FLp}
Let $d \ge 3$, $1 \le p <d$ and $1 \le \s \le \infty$. 
Suppose that the initial data $(a_0,m_0)$ satisfy 
$$
(a_0,m_0) 
\in \bigg(\fB{_{p,\s}^{-2+\frac{d}{p}}}\cap\fB{_{p,1}^{\frac{d}{p}}}\bigg)(\R^d) 
\times \bigg(\fB{_{p,\s}^{-3+\frac{d}{p}}}\cap\fB{^{-1+\frac{d}{p}}_{p,1}}\bigg)(\R^d). 
$$
There exists a positive constant $\ve_0 \ll 1$ such that if 
\begin{equation} \label{assump;initial_FB}
\|a_0\|_{\fB{_{p,\s}^{-2+\frac{d}{p}}}\cap\fB{_{p,1}^{\frac{d}{p}}}}
            +\|m_0\|_{\fB{_{p,\s}^{-3+\frac{d}{p}}}\cap\fB{_{p,1}^{-1+\frac{d}{p}}}}
\le \ve_0,  
\end{equation}
then the problem \eqref{eqn;mNSK} admits a unique global-in-time solution $(a,m)$ satisfying 
\begin{equation} \label{est;CL^p_T}
\|(a,m)\|_{CL_{T}^{(p,\s)}} 
\les \|(|\N|a_0,m_0)\|_{\fB{_{p,\s}^{-3+\frac{d}{p}}}\cap\fB{_{p,1}^{-1+\frac{d}{p}}}} 
\end{equation}
for all $T>0$. 
In particular, the solution $(a,m)$ fulfills $e^{\sqrt{c_0 t}|\N|}(a,m) \in CL_T^{(p,\s)}$ for all $T>0$, where 
$c_0>0$ is a constant 
and $e^{\sqrt{c_0 t}|\N|}f:=\F^{-1} e^{\sqrt{c_0 t}|\xi|} \F f$ for some 
$f \in \mathcal{S}'(\R^d)$.  
\end{thm}

\begin{thm}[{\it $L^p$-$L^1$ type decay estimates}\hspace{0mm}] \label{thm;Lp-L1}
Let $p$ satisfy $1 \le p \le 2$ and $\s=1$.   
Suppose that the initial data $(a_0,m_0)$ 
satisfy the same assumption as in Theorem \ref{thm;GWP_FLp} and 
$(a,m) \in CL_T^{(p,1)}$ denote the corresponding global-in-time solution of the problem \eqref{eqn;mNSK}. 
If in addition, we assume that $m_0=\N \tilde{m}_0$ and $(a_0,\tilde{m}_0)$ satisfies 
\begin{align} \label{cond;add_small}
\mathcal{D}_{p,0}:= \sup_{j \in \Z} 2^{-\frac{d}{p'}j}\|(\Dj a_0, \Dj\tilde{m}_0)\|_{\wh{L}^p} < \infty, 
\end{align}
then the global solution $(a,m)$ satisfies the following decay estimates:  
\begin{align} \label{est;decay1}
&\||\N|^{s_1}a(t)\|_{\fB{_{p,1}^{0}}}=O(t^{-\frac{d}{2}(1-\frac{1}{p})-\frac{s_1}{2}}) \;(t \to \infty) 
\quad \text{for all } s_1>-\frac{d}{p'}, \\
&\||\N|^{s_2}m(t)\|_{\fB{_{p,1}^{0}}}=O(t^{-\frac{d}{2}(1-\frac{1}{p})-\frac{s_2+1}{2}}) \;(t \to \infty)  
\quad \text{for all } s_2>-1-\frac{d}{p'}.   \label{est;decay2}
\end{align}
\end{thm}

\noindent
{\bf Comments on Theorems \ref{thm;GWP_FLp}, \ref{thm;Lp-L1}}  
\begin{enumerate}[1.]
\item ({\it Global well-posedness}) In the zero sound speed case, the fundamental solution $(a,m)$ of \eqref{eqn;mNSK3} satisfy the followings 
(as for the regorous formulation, you are able to see in Section \ref{sect;LA}):  
\begin{equation} \label{eqn;rough}
\mathcal{P}_\s m \simeq e^{\mu t \Delta} \mathcal{P}_\s m_0, 
\quad (|\N|a,\mathcal{P}_\s^\perp m) \simeq e^{\frac{\nu}{2}t \Delta} e^{\sqrt{\nu^2-4\kappa} t \Delta} (|\N|a_0, \mathcal{P}_\s^\perp m_0), 
\end{equation}
where $\mathcal{P}_\s$ is the Helmholtz projection which is defined by $\text{Id}-(-\Delta)^{-1}\N \div$ and $\mathcal{P}_\s^\perp$ is the orthogonal projection of 
$\mathcal{P}_\s$, namely $\mathcal{P}_\s^\perp=(-\Delta)^{-1}\N \div$. Focusing on the second relation in \eqref{eqn;rough}, 
we easily see that if $\nu^2<4 \kappa$, $(|\N|a,\mathcal{P}_\s^\perp m)$ 
is identified as the fundamental solution of the complex Ginzburg-Landau equations (see e.g., Doering-Gibbon-Levermore \cite{Do-Gi-Le}) with the same initial data 
$(|\N|a_0, \mathcal{P}_\s^\perp m_0)$ because of $e^{\sqrt{\nu^2-4\kappa} t \Delta}=e^{i\sqrt{4\kappa-\nu^2} t \Delta}$. 
In the Fourier side, it is straightforward that if $\nu^2<4\kappa$, 
$$
|(\wh{|\N|a},\wh{\mathcal{P}_\s^\perp m})| \simeq e^{-\frac{\nu}{2}t |\xi|^2} |e^{-i\sqrt{4\kappa-\nu^2}t|\xi|^2} (\wh{|\N|a_0}, \wh{\mathcal{P}_\s^\perp m_0})| 
\simeq |e^{-\frac{\nu}{2}t |\xi|^2} (\wh{|\N|a_0}, \wh{\mathcal{P}_\s^\perp m_0})|. 
$$
According to the above relations, 
namely the vanishment of the dispersive operator likes the free Schr\"odinder operator $e^{it\Delta}$, 
we can expect that the structure on the Fourier-Besov spaces allows us to construct the solution for the problem \eqref{eqn;mNSK3} 
in critical $L^p$ framework with $p \neq 2$. The same structure can be observed in the perturbed solution around the constant magnetic field $\bar{B} \in \R^3$ 
of Hall-MHD system (cf. \cite{Fu-Na}, \cite{Na}). 
To the best of author's knowledge, 
Theorem \ref{thm;GWP_FLp} is the first result concerning the existence of global solutions to \eqref{eqn;mNSK3} with $\nu^2<4\kappa$ 
in critical spaces based on $L^p(\R^d)$ framework with $p \neq 2$. 

\item ({\it Analyticity}) 
Starting with the pioneering work by Foias-Temam \cite{Fo-Te}, 
for the solution of the nonlinear equations in fluid mechanics 
which are classified as the parabolic type in various scaling critical spaces, 
the same estimate as treated in \cite{Fo-Te} was obtained by many authors 
(cf. \cite{Ba}, \cite{Ba-Bi-Ta}, \cite{Ch-Da-Xu}, \cite{Le}, \cite{So-Xu}).  
As the corollary of Theorem \ref{thm;GWP_FLp}, we easily obtain the analyticity of solution with respect to 
$x$-direction to employ Oliver-Titi's argument (see \cite[Theorem 6]{Ol-Ti}). 

In  Theorem \ref{thm;GWP_FLp}, the restriction $1 \le p <d$ is came from the nonlinear estimate which can be seen in Lemma \ref{lem;bil2}. 
This was also mentioned in Song-Xu \cite{So-Xu} and, in the case $\nu^2 \ge 4 \kappa$, they obtained the result on the analyticity in more general $L^q$-$L^p$ framework 
(see the definition of $E_T^{p,q}$ in their literature \cite{So-Xu}).  
Our result may be extended as the corresponding Fourier-Besov framework, however this paper deal only with the case $p=q$ in order to concentrate 
the derivation of the asymptotic formula as seen in the next section.

\item ({\it Decay estimates}) 
The assertion of Theorem \ref{thm;Lp-L1} means that the $L^p$-$L^1$ type decay estimate holds true for the global solution 
with the critical regularity obtained by Theorem \ref{thm;GWP_FLp}. As we stated before, the linearized solution 
behaves like the fundamental solution of the heat equation in the Fourier space. Hence, we can expect that 
the solution of \eqref{eqn;mNSK3} has the same decay rate as the heat kernel. 
In the proof of Theorem \ref{thm;Lp-L1} in Section \ref{sect;Decay}, applying the decay estimates for the linearized solution 
as similar method developed by Kozono-Ogawa-Taniuchi \cite{Ko-Og-Ta} and estimate on the analyticity in 
the negative Fourier-Besov spaces (see Proposition \ref{prop;negative}), we are possible to get the time-decay of 
$\wh{L}^p(\R^d)$-norm when the low frequency of initial data is assumed to be \eqref{cond;add_small} weaker than 
$L^1(\R^d)$. Here we notice that the following continuous embeddings hold true for any $d \ge 1$ and $1 \le p \le \infty$, 
$$
L^1(\R^d) \hr \wh{L}^1(\R^d) \simeq \fB{_{1,\infty}^0}(\R^d) \hr \fB{_{p,\infty}^{-\frac{d}{p'}}}(\R^d)
$$
(for the proofs, see Lemmas \ref{lem;sm} and \ref{lem;equi} below). 

In Theorem \ref{thm;Lp-L1}, the restrictions on $\s$ and $p$ are came from the nonlinear estimate for 
$\N (a^2 \wt{I}_P(a))$. As for the source, you are able to see in the proof of Proposition \ref{prop;negative}. 
\end{enumerate}

\subsection{Main result I\hspace{-0.5mm}I: Asymptotic behavior of solutions} \label{subsect;Thm2.3} 
In what follows, we assume $\r_*=1$ for simplicity.  
In order to state our main result on the asymptotic behavior as treated in Fujigaki-Miyakawa \cite{Fj-Mi}, let us introduce the functions concerning 
with the asymptotic profiles as follows: 
For $j$, $k=1,2,\cdots,d$, 
\begin{equation*}
G_1(t)=
\left\{
\begin{split}
&\frac{1}{2}(\nu+\sqrt{\nu^2-4\kappa})G_-(t)-\frac{1}{2}(\nu-\sqrt{\nu^2-4\kappa})G_+(t), 
&if\;\; \nu^2 \neq 4 \kappa, \\
&\Big(1-\frac{\nu}{2}t\Delta \Big)G_{\nu/2}(t), &if\;\; \nu^2 = 4 \kappa,
\end{split}
\right. 
\end{equation*}

\begin{equation*}
G_2(t)=
\left\{
\begin{split}
& G_-(t)-G_+(t) &if\;\;\nu^2 \neq 4 \kappa, \\
& -t \Delta G_{\nu/2}(t) &if\;\;\nu^2 = 4 \kappa, \\
\end{split}
\right. 
\end{equation*}

\begin{equation*}
\wt{G}^{(j,k)}_2(t)=
\left\{
\begin{split}
& \wt{G}^{(j,k)}_-(t)-\wt{G}^{(j,k)}_+(t) &if\;\;\nu^2 \neq 4 \kappa, \\
& -t \Delta \mathcal{R}_j \mathcal{R}_k G_{\nu/2}(t) &if\;\;\nu^2 = 4 \kappa, \\
\end{split}
\right. 
\end{equation*}
where $\mathcal{R}_j$ is the {\it Riesz operator} defined by $\mathcal{F}^{-1} i \xi_j |\xi|^{-1} \mathcal{F}$ 
for $j=1,2,\dots,d$ and 
\begin{gather*}
{\dsp G_{\pm}(t)=\F^{-1}\left[\frac{e^{\lam_{\pm}(\xi)t}}{\sqrt{\nu^2-4 \kappa}}\right]} \;\;with \;\;
\lam_\pm(\xi)=-\frac{\nu}{2}|\xi|^2\left(1 \pm \sqrt{1-\frac{4\kappa}{\nu^2}}\right), \\
\wt{G}^{(j,k)}_{\pm}(t)=\F^{-1}\left[\frac{e^{\lam_{\pm}(\xi)t}}{\sqrt{\nu^2-4 \kappa}} 
\cdot \frac{i \xi_j i \xi_k}{|\xi|^2}\right](=\mathcal{R}_j \mathcal{R}_k G_2(t)), \quad 
G_{\nu/2}(t)=\F^{-1}\Big[e^{-\frac{\nu}{2}|\xi|^2t}\Big]. 
\end{gather*}

\begin{thm}[{\it Asymptotic behavior}\hspace{0mm}\,] \label{thm;Asympt}
Let $p$ satisfy $1 < p \le 2$ and $\s=1$.   
Suppose that the initial data $(a_0,m_0)$ 
satisfy the same assumption as in Theorem \ref{thm;GWP_FLp} and 
$(\r,m)$ denote the corresponding global-in-time solution of the problem \eqref{eqn;mNSK}. 
If in addition, we assume that $m_0=\N \tilde{m}_0$ and $(a_0,\tilde{m}_0)$ satisfies 
$(a_0,\tilde{m}_0) \in L^1(\R^d)$, 
then the global solution $(a,m)$ satisfies the following decay estimates for all $s > -d/p'$, 

\begin{equation} \label{est;AE}
\begin{aligned}
\lim_{t \to \infty}
t^{\frac{d}{2}(1-\frac{1}{p})+\frac{s}{2}}
\bigg\|
 |\N|^{s} 
 \bigg[
  a(t)&-G_1(t)\int_{\R^d} a_0(y) dy-G_2(t)\int_{\R^d} \tilde{m}_0(y) dy \\
      &-G_2(t) 
      \int_0^\infty 
      \int_{\R^d} 
       \bigg(\sum_{n=2}^\infty \frac{P^{(n)}(1)}{n!}a^n\bigg)
       \,dyd\t \\
      &-\wt{G}_2^{(j,k)}(t)
      \int_0^\infty \int_{\R^d} \Big(\frac{m_j m_k}{1+a}+\wt{\mathcal{K}}^{(j,k)}(a)\Big)
      \,dyd\t 
\bigg] \bigg\|_{\wh{L}^p}=0,  
\end{aligned} 
\end{equation}
where $\wt{\mathcal{K}}^{(j,k)}(a)$ is the $(j,k)$-th component of a symmetric $d \times d$ matrix which is defined by 
\begin{equation*}
\wt{\mathcal{K}}^{(j,k)}(a)
=
\left\{
\begin{split}
&\;\kappa \pt_j a \pt_k a  &if\;\;j \neq k, \\
&\;\frac{\kappa}{2} |\N a|^2 + \kappa (\pt_j a)^2  &if\;\;j = k.  \\
\end{split}
\right. 
\end{equation*}
Here we employ the summation convention with respect to $j$, $k=1,2,\dots,d$ in \eqref{est;AE}.  
\end{thm}
\begin{rem}
Under the same assumption as in Theorem \ref{thm;Asympt}, we are able to obtain that 
for all $s > -1-d/p'$, 
\begin{align} 
&\begin{aligned} \label{est;AE2}
\lim_{t \to \infty}
t^{\frac{d}{2}(1-\frac{1}{p})+\frac{s+1}{2}}
\bigg\|
 |\N|^{s} 
 \bigg[
  \mathcal{P}_\s m_j(t) 
      \,-\,&\pt_k S_\mu^{(j,k)}(t)  \\
      \times& \int_0^\infty \int_{\R^d} \Big(\frac{m_j m_k}{1+a}+\wt{\mathcal{K}}^{(j,k)}(a)\Big)
      \,dyd\t 
\bigg] \bigg\|_{\wh{L}^p}=0,    
\end{aligned} \\  
&\begin{aligned} \label{est;AE3}
\lim_{t \to \infty}
t^{\frac{d}{2}(1-\frac{1}{p})+\frac{s+1}{2}}
\bigg\|
 |\N|^{s} 
 \bigg[
  &\mathcal{P}_\s^\perp m_j(t)-\pt_j G_2(t) \int_{\R^d} a_0(y) dy-\pt_j G_3(t)\int_{\R^d} \tilde{m}_0(y) dy \\
      &-\pt_j G_3(t)
      \int_0^\infty 
      \int_{\R^d} 
       \bigg(\sum_{n=2}^\infty \frac{P^{(n)}(1)}{n!}a^n \bigg)
       \,dyd\t \\
      &-\pt_k \wt{G}^{(j,k)}_3(t) 
      \int_0^\infty \int_{\R^d} \Big(\frac{m_j m_k}{1+a}+\wt{\mathcal{K}}^{(j,k)}(a)\Big)
      \,dyd\t 
\bigg] \bigg\|_{\wh{L}^p}=0,     
\end{aligned}
\end{align}
where, in \eqref{est;AE2} and \eqref{est;AE3}, we employ the summation convention with respect to $k=1,2,\dots,d$, $\mathcal{P}_\s v_j$, $\mathcal{P}_\s^\perp v_j$ are denoted by $v_j +(-\Delta)^{-1} \pt_j \div v$, 
$-(-\Delta)^{-1} \pt_j \div v$, respectively and 
\begin{equation*}
S_\mu^{(j,k)}(t):=\mathcal{F}^{-1}\left[e^{- \mu t |\xi|^2} 
\left(\dl_{jk}-\frac{i \xi_j i \xi_k}{|\xi|^2}\right) \right], 
\end{equation*}
\begin{equation*}
G_3(t)=
\left\{
\begin{split}
&\frac{1}{2}(\nu+\sqrt{\nu^2-4\kappa}) G_+(t)-\frac{1}{2}(\nu-\sqrt{\nu^2-4\kappa}) G_-(t), 
&if\;\; \nu^2 \neq 4 \kappa, \\
&\Big(1-\frac{\nu}{2}t\Delta \Big) G_{\nu/2}(t), &if\;\; \nu^2 = 4 \kappa,
\end{split}
\right. 
\end{equation*}
\begin{equation*}
\wt{G}_3^{(j,k)}(t)=
\left\{
\begin{split}
&\frac{1}{2}(\nu+\sqrt{\nu^2-4\kappa}) \wt{G}^{(j,k)}_+(t)-\frac{1}{2}(\nu-\sqrt{\nu^2-4\kappa})
 \wt{G}^{(j,k)}_-(t) , 
&if\;\; \nu^2 \neq 4 \kappa, \\
&\Big(1-\frac{\nu}{2}t\Delta \Big) \mathcal{R}_j \mathcal{R}_k G_{\nu/2}(t), &if\;\; \nu^2 = 4 \kappa. 
\end{split}
\right. 
\end{equation*}
\end{rem}

This paper is organized as follows: After some preparations on the elemental properties and estimates for the linearized solution on the 
Fourier-Besov spaces in Sections \ref{sect;prelim}, \ref{sect;LA}, we give the proof of Theorem \ref{thm;GWP_FLp} in Section \ref{sect;GWP}. 
In Section \ref{sect;Decay}, we give the proof of the uniform estimate in negative Fourier-Besov settings and $L^p$-$L^1$ type decay estimate as applying 
the estimate on the analyticity. Section \ref{sect;AP} is devoted to the proof of Theorem \ref{thm;Asympt}. 

\vskip2mm 
\noindent
{\bf Notation.}\;Throughout this paper, $C$ stands for a generic constant 
(the meaning of which depends on the context). 
Let $X$ be a Banach space, $I \subset \R$ be an interval,  
and $1\le r \le \infty$. One denotes by the Bochner space   
$L^r(I;X)$ the set of strongly measurable functions $f:I\to X$ such that $t \mapsto \|f(t)\|_X$ 
belongs to $L^r(I)$. For $f \in L^r(I;X)$, one defines $\|f\|_{L^r(I;X)}=\|\|f\|_{X}\|_{L^r(I)}$.   
For simplicity, we denote the Chemin--Lerner space $\wt{L^r(I;}\fB{_{p,\s}^s})$ 
by $\wt{L^r_T}(\fB{_{p,\s}^s})$. 
For $s \in \R$ and $f \in \mathcal{S}'$, $|\N|^s$ is designating the Riesz potential defined by $|\N|^s f:=\mathcal{F}^{-1}[|\xi|^s \wh{f}\,]$ and 
$\langle \N \rangle^s f:=\mathcal{F}^{-1}[\langle \xi\rangle^s \wh{f}]$ with $\langle\xi\rangle^s:=(1+|\xi|^2)^{s/2}$. 
For any normed space $Y$, we denote by 
$\|(f,g)\|_Y := (\|f\|_Y^2 + \|g\|_Y^2)^{1/2}$ with $f$, $g \in Y$. 
In what follows, $f_j:=\phi_j*f$, $S_{j}f:= \sum_{k \le j-1} f_k$, 
$\tilde{\phi}_j:=\sum_{|j-k|\le 1} \phi_k$ and $\tilde{f}_j:= \tilde{\phi}_j*f$. 
For $1 \le q<\infty$, $\ell^q(\Z)$ denotes the set of all sequences $\{a_j\}_{j \in \Z}$ of real numbers such that $\sum_{j \in \Z}|a_j|^q$ converges. 
In the case $q=\infty$, let $\ell^\infty(\Z)$ denote the set of all bounded sequences of real numbers.  

\sect{Preliminaries}\label{sect;prelim}

In the followings, we present the basic properties with regards to 
the Fourier--Besov spaces. First of all, let us now recall Bernstein's inequalities   
which allows us to obtain some embedding of spaces (for the proofs, see Lemma A.1 in \cite{Ch2}).   
\begin{lem}[{\it Bernstein-type lemma} \cite{Ch2}] \label{lem;Bern}
Let $d \ge 1$, $1 \le p \le q \le \infty$, $\lam$, $R >0$ and $0 < R_1 < R_2$.  
For any $s \in \R$, there exists some constant $C>0$ such that 
\begin{align*}
  &\||\N|^s f\|_{\wh{L}^q}
  \le C \lam^{s + d(\frac{1}{p}-\frac{1}{q})}\|f\|_{\wh{L}^p} \quad 
  if \;\; \supp \wh{f} \subset \{\xi \in \R^d;|\xi| \le \lam R\},\\
  &C^{-1} \lam^s \|f\|_{\wh{L}^p}
  \le \||\N|^s f\|_{\wh{L}^p}
  \le C \lam^s \|f\|_{\wh{L}^p} \quad 
  if \;\; \supp \wh{f} \subset \{\xi \in \R^d;\lam R_1 \le |x| \le \lam R_2\}, 
\end{align*}
where the Fourier--Lebesgue space $\wh{L}^p=\wh{L}^p(\R^d)$ is defined by
$$
  \wh{L}^p(\R^d)=\{f \in \mathcal{S}'; \wh{f} \in L^1_{loc}(\R^d),\|f\|_{\wh{L}^p} <\infty\} \quad \text{with} \quad
  \|f\|_{\wh{L}^p}=\|\wh{f}\|_{L^{p'}}. 
$$
\end{lem}

\begin{lem}[{\it Sobolev-type embedding \cite{Na}}\hspace{0mm}] \label{lem;sm}
Let $s \in \R$, $d \ge 1$, $1 \le p_1 \le p_2 \le \infty$ and $1 \le \s_1\le\s_2 \le \infty$.  
Then the following continuous embedding holds:
\begin{align*}
  \fB{_{p_1,\s_1}^s}(\R^d) \hr \fB{_{p_2,\s_2}^{s-d(\frac{1}{p_1}-\frac{1}{p_2})}}(\R^d). 
\end{align*}
\end{lem}

\begin{lem}[\cite{Ch2}, \cite{Ko-Yo}] \label{lem;equi}
Let $s \in \R$, $d \ge 1$ and $1 \le p \le \infty$. The spaces $\fB{_{p,p'}^s}(\R^d)$ and 
$\wh{\dot{H}}{_p^s}(\R^d)$ satisfy 
$
\fB{_{p,p'}^s}(\R^d) \simeq \wh{\dot{H}}{_p^s}(\R^d)
$ 
in the sence of norm equivalence. 
Here the Fourier-Sobolev space $\wh{\dot{H}}{_p^s}=\wh{\dot{H}}{_p^s}(\R^d)$ is defined by 
$$
  \wh{\dot{H}}{_p^s}(\R^d)
  =\{f \in \mathcal{S}'; \wh{f} \in L^1_{loc}(\R^d),\|f\|_{\wh{\dot{H}}{_p^s}}<\infty\} \quad \text{with} \quad
  \|f\|_{\wh{\dot{H}}{_p^s}}=\||\N|^sf\|_{\wh{L}^p}. 
$$
\end{lem}

\subsection{Bilinear estimates and product estimates}
Let us present the various types of product estimates in 
Fourier-Besov spaces or Chemin-Lerner type spaces. 
In the following lemmas, we state the standard bilinear estimates without their proof 
(for the general statement and their proof, see e.g., Lemma 2.4 in \cite{Ma-Na-Og}).  
\begin{lem}[{\it Bilinear estimates}\hspace{0mm}] \label{lem;bil}
Let $s > 0$ and $1 \le p,\s \le \infty$. 
There exists some constant $C > 0$ such that the following estimate holds: 
\begin{equation} \label{est;prod_bl}
\|fg\|_{\fB{_{p,\s}^s}}
\le 
C\left(\|f\|_{\wh{L}^\infty}\|g\|_{\fB{^s_{p,\s}}} 
+ \|f\|_{\fB{^s_{p,\s}}}\|g\|_{\wh{L}^\infty}\right). 
\end{equation}
\end{lem}   

The following product estimate in Besov spaces is well-known as in \cite{Ab-Pa}. 
The similar estimate in Fourier-Besov spaces is established in \cite{Na}. 
\begin{lem}[{\it Product estimates} \cite{Na}] \label{lem;A-P}
Let $d \ge 1$ and $1 \le p,\s \le \infty$.  
If $s \in \R$ satisfying $|s|<\frac{d}{p}$ for $2 \le p$ and $-\frac{d}{p'}<s<\frac{d}{p}$ for $1\le p<2$, 
then there exists some constant $C>0$ such that 
\begin{equation} \label{est;prod_A-P2}
\|fg\|_{\fB{_{p,\s}^{s}}}
\le C\|f\|_{\fB{_{p,\s}^s}} 
     \|g\|_{\wh{L}^{\infty}\cap\fB{_{p,\infty}^\frac{d}{p}}}. 
\end{equation}
\end{lem}

\begin{lem}[{\it Product estimates I\hspace{-0.4mm}I} \cite{Na}] \label{lem;prod_est}
Let $s \in \R$, $1 \le p,\s \le \infty$, $I=(0,T)$ with $T \in \R_+$, 
$1 \le r,r_i \le \infty$ $(i=1,2,3,4)$ satisfying 
$\frac{1}{r}=\frac{1}{r_1}+\frac{1}{r_2}=\frac{1}{r_3}+\frac{1}{r_4}$.  
\begin{enumerate}
\item If $s>0$, there exists some constant $C>0$ such that 
\begin{equation} \label{est;prod_1-1}
\|fg\|_{\wt{L^r(I};\fB{_{p,\s}^{s}})}
\le C\left(\|f\|_{\wt{L^{r_1}(I};\fB{_{p,\s}^s})} 
           \|g\|_{L^{r_2}(I;\wh{L}^{\infty})}
         +\|f\|_{L^{r_3}(I;\wh{L}^\infty)} 
          \|g\|_{\wt{L^{r_4}(I};\fB{_{p,\s}^s})}\right). 
\end{equation}
\item 
If $s \in \R$ satisfying $|s|<\frac{d}{p}$ for $2 \le p$ and $-\frac{d}{p'}<s<\frac{d}{p}$ for $1\le p<2$, 
then there exists some constant $C>0$ such that 
\begin{equation} \label{est;prod_1-2}
\|fg\|_{\wt{L^r(I};\fB{_{p,\s}^{s}})}
\le C\|f\|_{\wt{L^{r_1}(I};\fB{_{p,\s}^s})} 
     \|g\|_{ L^{r_2}(I;\wh{L}^{\infty}) \cap \wt{L^{r_2}(I};\fB{_{p,\infty}^\frac{d}{p}})}. 
\end{equation}
\end{enumerate}
\end{lem}

\begin{rem}
The proof of Lemma \ref{lem;prod_est} (1) is same as the one of Lemma 2.4 in \cite{Ma-Na-Og}. 
By \eqref{est;prod_bl}, \eqref{est;prod_1-1} and embeddings 
$\fB{_{p,1}^{d/p}}(\R^d) \hr \fB{_{\infty,1}^0}(\R^d) \simeq \wh{L}^\infty(\R^d)$, 
we obtain 
\begin{align} \label{est;Banach_ring}
\|fg\|_{\cfB^{\frac{d}{p}}}
\les \|f\|_{\cfB^\frac{d}{p}}\|g\|_{\cfB^\frac{d}{p}}, \quad 
\|fg\|_{\wt{L^\infty(I};\cfB^\frac{d}{p})}
\les \|f\|_{\wt{L^\infty(I};\cfB^\frac{d}{p})}\|g\|_{\wt{L^\infty(I};\cfB^\frac{d}{p})}. 
\end{align}
\end{rem}

The following estimate in Besov spaces is recently obtained by Song-Xu \cite{So-Xu}. 
In this paper, we establish the corresponding bilinear estimate in the Fourier-Besov framework. 
As for the proof, you are able to see in \S \ref{sect;appendix}. 
\begin{lem}[{\it Bilinear estimates I\hspace{-0.4mm}I}\hspace{1mm}\hspace{0mm}] \label{lem;bil2}
Let $d \ge 3$, $1 \le p <d$, $1 \le \s \le \infty$, $T>0$ and $1 \le r,r_i \le \infty$ $(i=1,2,3,4)$ satisfying 
$\frac{1}{r}=\frac{1}{r_1}+\frac{1}{r_2}=\frac{1}{r_3}+\frac{1}{r_4}$. 
There exists some positive constant $C>0$ such that 
\begin{align*}
\|f g\|_{\wt{L^r_T(}\fB{_{p,\s}^{-2+\frac{d}{p}}})}
\le C\left(\|f\|_{\wt{L^{r_1}_T(}\fB{_{p,\s}^{-2+\frac{d}{p}}})}
\|g\|_{\wt{L^{r_2}_T(}\fB{_{p,\infty}^{\frac{d}{p}}})}
+\|f\|_{\wt{L^{r_3}_T(}\fB{_{p,\infty}^{\frac{d}{p}}})}
\|g\|_{\wt{L^{r_4}_T(}\fB{_{p,\s}^{-2+\frac{d}{p}}})} \right). 
\end{align*}
\end{lem}

\begin{lem}[{\it The estimate for} $\B_t(f,g)$ \cite{Na3}] \label{lem;B_t}
For any $1 \le p$, $p_1$, $p_2 \le \infty$ with $\frac{1}{p}=\frac{1}{p_1}+\frac{1}{p_2}$, 
there exists some constant $C>0$ such that 
the following estimate holds true for all $t \ge 0$: 
$$
\|\B_t(f,g)\|_{\widehat{L}^p} 
\le C \|f\|_{\widehat{L}^{p_1}}\|g\|_{\widehat{L}^{p_2}}, 
$$
where $\B_t(f,g)=\B_t(f,g)(t,x):=e^{\sqrt{c_0 t}|\N|}(e^{-\sqrt{c_0 t}|\N|}f e^{-\sqrt{c_0 t}|\N|}g)(x)$ 
with some constant $c_0>0$.  
\end{lem}

\subsection{Smoothing properties in Chemin-Lerner spaces}
Let us consider the following inhomogeneous heat equation with a diffusive coefficient $\nu > 0$:
\begin{equation} \label{eqn;heat}
\left\{
\begin{split}
  &\pt_t u - \nu \Delta u = f, &t>0, \,x \in \R^d, \\
  &u|_{t=0} = u_0, &x \in \R^d.  
\end{split}
\right.
\end{equation}

\begin{lem}[{\it Smoothing estimates for} \eqref{eqn;heat} \cite{Ma-Na-Og}, \cite{Na}] \label{lem;MR}
Let $s \in \R$, $d \ge 1$, $1 \le p,\s \le \infty$, $1 \le r_1 \le r \le \infty$ and $I:=(0,T)$ with $T \in \R_+$. 
Suppose that $u_0 \in \fB{^s_{p,\s}}$ and $f \in \wt{\,L^{r_1}(I};\wh{\dB}{_{p,\s}^{s-2+2/r_1}})$. 
The inhomogeneous heat equation \eqref{eqn;heat} has a unique solution $u$ and 
there exists some constant $C = C(r)>0$ such that 
\begin{equation*} 
\nu^\frac{1}{r} \|u\|_{\wt{L^r(I;}\wh{\dB}{_{p,\s}^{s+\frac{2}{r}}})}
\le C \left(\|u_0\|_{\wh{\dB}{_{p,\s}^s}} 
      + \nu^{-1+\frac{1}{r_1}} \|f\|_{\wt{L^{r_1}(I};\wh{\dB}{_{p,\s}^{s-2+\frac{2}{r_1}}})} \right). 
\end{equation*} 
\end{lem}

\subsection{Culculus facts}
The following two inequalities play crucial role in the proof of decay estimates 
in Theorem \ref{thm;Lp-L1} (for the proofs, see e.g., \cite{Ch-Da}). 
\begin{lem} \label{lem;ce}
For any $a$, $b > 0$ with $\max (a,b)>1$, there exists a positive constant $C$ such that 
$$
\int_0^t
\langle t - \t \rangle^{-a} \langle \t \rangle^{-b}
d\t
\le C \langle t \rangle^{- \min(a,b)} \quad \text{for all } t \ge 0. 
$$
\end{lem}

\begin{lem}[{\it Page 73 in} \cite{B-C-D}] \label{lem;cineq}
Let $\delta_0>0$. For all $\s>0$, then there exists a constant $C_\s>0$ depending only on $\s$ such that 
$$
\sup_{t\ge 0}\sum_{j\in\Z}\big(t^\frac{1}{2}2^j\big)^\s e^{-t\delta_02^{2j}} \le C_\s. 
$$
\end{lem}

\sect{Linear Analysis} \label{sect;LA}

In what follows, we assume $\r_*=1$ for simplisity and $a:=\r-1$. We rapidly obtain the following equations for the 
perturbation denoted by $(a, m)$:  
\begin{equation} \label{eqn;mNSK2} 
\left\{
\begin{split}
& \pt_ta+\div m=0, 
&t>0, x \in \R^d, \\
&\pt_t m -\mathcal{L}m-\kappa \N \Delta a = \mN(a,m)
&t>0, x \in \R^d, \\
&(a, m)|_{t=0}=(a_0,m_0), &x \in \R^d,   
\end{split}
\right.
\end{equation}
where the nonlinear part $\mN(a,m)$ is given by 
\begin{gather} 
\mN(a,m)=\text{div} \left((I(a)-1) m \otimes m\right)-I_P(a)\N a 
- \mathcal{L}(I(a) m)+ \div \mathcal{K}(a), \label{eqn;nonlin} \\ 
I(a):=\frac{a}{1+a}, \quad 
I_P(a):=P'(1+a).      \label{eqn;I}
\end{gather}

In this section, we present the pointwise behavior and regularizing estimates for the solution to the following linearized system 
at the Fourier side:  
\begin{equation} \label{eqn;L} 
\left\{
\begin{split}
& \pt_t \wh{a}+i \xi \cdot \wh{m}=\wh{f}, 
&t>0, x \in \R^d, \\
&\pt_t \wh{m} +\mu |\xi|^2 \wh{m}+(\lam+\mu)\xi (\xi \cdot \wh{m})+i \kappa \xi |\xi|^2 \wh{a} = \wh{g}
&t>0, x \in \R^d, \\
&(\wh{a}, \wh{m})|_{t=0}=(\wh{a}_0,\wh{m}_0), &x \in \R^d.   
\end{split}
\right.
\end{equation}

\begin{prop}[{\it Pointwise estimate in the Fourier space} \cite{Ch-Ko}, \cite{Na2}\hspace{0mm}] \label{lem;pw_L}
Let $(a,m)$ be a solution to the problem \eqref{eqn;L}. 
There exists some positive constant $C>0$ and $c_0>0$ 
such that   
\begin{align} \label{est;ptw}
\big|\big(|\xi|\wh{a},\wh{m}\big)(t,\xi)\big| 
\le Ce^{-tc_0|\xi|^2}\big|\big(|\xi|\wh{a},\wh{m}\big)(0,\xi)\big| 
+C\int_0^t e^{-(t-\t)c_0|\xi|^2}\big|\big(|\xi|\wh{f},\wh{g}\big)\big|\,d\t. 
\end{align}
\end{prop}

\begin{prop}[{\it Regularizing estimates for \eqref{eqn;L}} \cite{Ch-Ko}, \cite{Na2}\hspace{0mm}] \label{prop;MR3}
Let $s \in \R$, $1 \le p,\s \le \infty$, $1 \le r_1 \le r \le \infty$ and $I=(0,T)$ with $T \in \R_+$. 
The inhomogeneous problem \eqref{eqn;L} has a unique solution $(\r,u,B)$ and 
there exists some constant $C >0$ such that 
\begin{equation*} 
\|(|\N|a,m)\|_{\wt{L^r(I;}\wh{\dB}{_{p,\s}^{s+\frac{2}{r}}})}
\le C \left(\|(|\N|a_0,m_0)\|_{\wh{\dB}{_{p,\s}^s}} 
      +\|(|\N|f,g)\|_{\wt{L^{r_1}(I;}\wh{\dB}{_{p,\s}^{s-2+\frac{2}{r_1}}})} \right). 
\end{equation*} 
\end{prop}

In the case of $\nu^2 \neq 4 \kappa$, by solving the linearized equation obtained by the homogeneous problem, 
namely \eqref{eqn;L} with $f$, $g\equiv 0$ (see, e.g., \S\,3 in \cite{Ko-Ts}), we see that 
the solution $(a,m)$ of the homogeneous problem can be written down explicitly as followings: 
\begin{align*}
&\wh{a}(t,\xi) 
=\frac{\lam_+(\xi)e^{\lam_-(\xi)t}-\lam_-(\xi)e^{\lam_+(\xi)t}}{\lam_+(\xi)-\lam_-(\xi)} \wh{a}_0(\xi)
  -i\,{}^t\xi \frac{e^{\lam_+(\xi)t}-e^{\lam_-(\xi)t}}{\lam_+(\xi)-\lam_-(\xi)} \wh{m}_0(\xi), \\
&\wh{m}(t,\xi)
= e^{-\mu|\xi|^2t}\wh{m}_0(\xi)-i\kappa\xi|\xi|^2\frac{e^{\lam_+(\xi)t}-e^{\lam_-(\xi)t}}{\lam_+(\xi)-\lam_-(\xi)}\wh{a}_0(\xi) \\
&\qquad \qquad \qquad \qquad \qquad  +\left(\frac{\lam_+(\xi)e^{\lam_+(\xi)t}-\lam_-(\xi)e^{\lam_-(\xi)t}}{\lam_+(\xi)-\lam_-(\xi)}-e^{-\mu|\xi|^2t}\right) 
  \frac{\xi(\xi\cdot\wh{m}_0(\xi))}{|\xi|^2}, 
\end{align*}
where $\lam_\pm(\xi)=-\frac{\nu}{2}|\xi|^2(1 \pm \sqrt{1-\frac{4\kappa}{\nu^2}})$ are the solution of the corresponding characteristic equation 
$\lam^2+\nu|\xi|^2\lam+\kappa |\xi|^4=0$.  
Now, we set $\mathcal{G}^{i,j}=\mG^{i,j}(t,\xi)$ $(i,j=1,2)$ as followings: 
\begin{equation}  \label{eqn;neq}
\begin{aligned}
&\mG^{1,1}(t,\xi):=\frac{\lam_+(\xi)e^{\lam_-(\xi)t}-\lam_-(\xi)e^{\lam_+(\xi)t}}{\lam_+(\xi)-\lam_-(\xi)}, \quad 
\mG^{1,2}(t,\xi):=i\,{}^t\xi \frac{e^{\lam_+(\xi)t}-e^{\lam_-(\xi)t}}{\lam_+(\xi)-\lam_-(\xi)}, \\
&\mG^{2,1}(t,\xi):={-i\kappa \xi|\xi|^2\frac{e^{\lam_+(\xi)t}-e^{\lam_-(\xi)t}}{\lam_+(\xi)-\lam_-(\xi)}}, \\
&\mG^{2,2}(t,\xi):={\frac{\lam_+(\xi)e^{\lam_+(\xi)t}-\lam_-(\xi)e^{\lam_-(\xi)t}}{\lam_+(\xi)-\lam_-(\xi)} \cdot \frac{\xi \otimes \xi}{|\xi|^2}
                             +e^{-\mu|\xi|^2t}\left(\text{Id}-\frac{\xi \otimes \xi}{|\xi|^2}\right)},   
\end{aligned}
\end{equation}
where ${}^t \xi$ is the transpose of $\xi$. 
Hereafter, we also set $G^{i,j}(t,x):=\F^{-1}[\mG^{i,j}(t,\cdot)]$ $(i,j=1,2)$ and 
\begin{align*}
\mG(t,\xi)
:=
\left[
\begin{array}{@{\,}cc@{\,}}  
\mG^{1,1}(t,\xi) &\mG^{1,2}(t,\xi) \\
\mG^{2,1}(t,\xi) &\mG^{2,2}(t,\xi)
\end{array}
\right], \quad 
G(t,x):=
\left[
\begin{array}{@{\,}cc@{\,}}  
G^{1,1}(t,x) &G^{1,2}(t,x) \\
G^{2,1}(t,x) &G^{2,2}(t,x)
\end{array}
\right].  
\end{align*}

On the other hand, in the case of $\nu^2 = 4\kappa$, $\mathcal{G}^{i,j}(t,\xi)$ are given by followings instead of \eqref{eqn;neq}: 
\begin{equation}  \label{eqn;neq2}
\begin{aligned}
&\mG^{1,1}(t,\xi):=\left(1+\frac{\nu}{2} t|\xi|^2 \right)e^{-\frac{\nu}{2}t|\xi|^2}, \quad 
\mG^{1,2}(t,\xi):=- t e^{-\frac{\nu}{2}t|\xi|^2} i\,{}^t\xi, \\
&\mG^{2,1}(t,\xi):=-i\kappa \xi t |\xi|^2  e^{-\frac{\nu}{2}t|\xi|^2}, \quad 
\mG^{2,2}(t,\xi):=\left(1-\frac{\nu}{2} t|\xi|^2 \right) e^{-\frac{\nu}{2}t|\xi|^2}.    
\end{aligned}
\end{equation}

Let $U(t):={}^t(a,m)$ be solution to \eqref{eqn;mNSK2} with $U_0:={}^t(a_0,m_0)$. 
Since the solution of \eqref{eqn;L} is given by Green matrix $G$, we see that $U(t)$ is represented by 
\begin{equation} \label{eqn;IE}
\begin{aligned}
U(t)=&G(t,\cdot)*U_0
+\int_0^t 
G(t-\t,\cdot)* 
\left(
\begin{array}{@{\,}c@{\,}}
0\\
\mathcal{N}(a,m)
\end{array}
\right)
d\t \\ 
=& 
\left(
\begin{array}{@{\,}c@{\,}}
G^{1,1}(t,\cdot)*a_0+G^{1,2}(t,\cdot)*m_0+\dsp \int_0^t G^{1,2}(t-\t,\cdot)*\mathcal{N}(a,m)\,d\t\\
\dsp G^{2,1}(t,\cdot)*a_0+G^{2,2}(t,\cdot)*m_0+\int_0^t G^{2,2}(t-\t,\cdot)*\mathcal{N}(a,m)\,d\t
\end{array}
\right). 
\end{aligned}
\end{equation}


\sect{Existence of the global solution and analiticity} \label{sect;GWP}

\subsection{Existence of global solutions} 

In this section, we give the proof of Theorem \ref{thm;GWP_FLp}.  
In order to prove Theorem \ref{thm;GWP_FLp}, we apply Banach's fixed point argument on 
the complete metric space $CL_T^{(p,\s)}$ as given in Definition \ref{dfn;CL_T}. 
Let us define the solution mapping $\Phi=\Phi[(b,n)]$ as follows:  
Given initial data such that $(a_0,m_0)$ satisfy the assumptions as in Theorem \ref{thm;GWP_FLp}, 
the solution mapping $\Phi$ is defined by 
\begin{equation}
\Phi: (b,n) \longmapsto (a,m)
\end{equation} 
with $(a,m)$ the solution to 
\begin{equation} \label{eqn;AE} 
\left\{
\begin{split}
&\;\pt_t a+\div m=0, \\
&\;\pt_t m-\L m -\kappa\N \Del a =\mathcal{N}(b,n), \\
&\;(a, m)|_{t=0}=(a_0,m_0), 
\end{split}
\right.
\end{equation}
where $\mathcal{N}(b,n)=\mathcal{N}(a,m)|_{(a,m)=(b,n)}$ are defined in \eqref{eqn;nonlin}. 

\begin{lem} \label{lem;nonlin_est}
Let $d \ge 3$, $1 \le p < d$, $1 \le \s \le \infty$ and $T>0$. Let $\mathcal{N}$ be defined in \eqref{eqn;nonlin} and 
$P$ be a real analytic function in a neighborhood of $1$ such that $P'(1)=0$. 
There exists some positive constant $C>0$ such that 
\begin{equation} \label{est;nonlin_all}
\begin{aligned}
\|\mathcal{N}(b,n)\|_{\wt{L^1_t(}\fB{_{p,\s}^{-3+\frac{d}{p}}}) \cap L^1_t(\fB{_{p,1}^{-1+\frac{d}{p}}})}
\le C(\|(b,n)\|_{CL^{(p,\s)}_T}^2+\|(b,n)\|_{CL^{(p,\s)}_T}^3)
\end{aligned}
\end{equation}
for all $(b,n) \in CL^{(p,\s)}_T$ and $0 \le t \le T$. 
\end{lem}

\begin{proof}[The proof of Lemma \ref{lem;nonlin_est}] 
The following estimate for $\|\mathcal{N}(b,n)\|_{L^1_t(\fB{_{p,1}^{-1+\frac{d}{p}}})}$ 
is already obtained by Lemma 5.1 in \cite{Ko-Na}:
\begin{align*}
\|\mathcal{N}(b,n)\|_{L^1_t(\fB{_{p,1}^{-1+\frac{d}{p}}})} \les \|(b,n)\|_{CL^p_T}^2+\|(b,n)\|_{CL^p_T}^3, 
\end{align*}
where $\|(b,n)\|_{CL^p_T}:=\|(\langle\N\rangle b,n)\|_{\wt{L^\infty_t(}\fB{_{p,1}^{-1+d/p}}) 
\cap L^1_t(\fB{_{p,1}^{1+d/p}})}$. Because $\|(b,m)\|_{CL^p_T} \les \|(b,m)\|_{CL^{(p,\s)}_T}$, 
we could obtain the desired estimate for the high frequencies of solution as follows: 
\begin{equation} \label{est;nonlin_h}
\|\mathcal{N}(b,n)\|_{L^1_t(\cfB^{-1+\frac{d}{p}})} 
\les \|(b,n)\|_{CL^{(p,\s)}_T}^2+\|(b,n)\|_{CL^{(p,\s)}_T}^3. 
\end{equation}

Let us consider the estimate for $\|\mathcal{N}(b,n)\|_{\wt{L^1_t(}\fB{_{p,\s}^{-3+\frac{d}{p}}})}$. 

\noindent
{\it The estimate for $\div ((I(b)-1)\,n \otimes n)$}: By Lemma \ref{lem;Bern} and Lemma \ref{lem;bil2}, 
we obtain   
\begin{align} \label{est;nonlin1}
\|\div (n \otimes n)\|_{\wt{L^1_t(}\fB{_{p,\s}^{-3+\frac{d}{p}}})}
\les \|n\otimes n\|_{\wt{L^1_t(}\fB{_{p,\s}^{-2+\frac{d}{p}}})}
\les \|n\|_{\wt{L^2_t(}\fB{_{p,\s}^{-2+\frac{d}{p}}})}
\|n\|_{\wt{L^2_t(}\fB{_{p,\infty}^{\frac{d}{p}}})}.  
\end{align}  

Thanks to the Taylor expansion, we see $(1+b)^{-1}=\sum_{n=0}^\infty (-1)^n b^n$ if $|b|<1$. 
By using Lemma \ref{lem;bil2}, \eqref{est;Banach_ring} and $\|b\|_{\wt{L^\infty_t(}\fB{_{p,1}^\frac{d}{p}})} \le \ve_0 \ll 1$, 
we see that for all $1 \le r,\s \le \infty$, 
\begin{equation} 
\begin{aligned} \label{est;1/1+r}
\left\|\frac{b}{1+b}\right\|_{\wt{L^r_t(}\fB{_{p,\s}^{-2+\frac{d}{p}}})} 
\le& \|b\|_{\wt{L^r_t(}\fB{_{p,\s}^{-2+\frac{d}{p}}})} \sum_{n=0}^\infty \bigg(C\|b\|_{L^\infty_t(\fB_{p,\infty}^\frac{d}{p})}\bigg)^n \\
\le& \|b\|_{\wt{L^r_t(}\fB{_{p,\s}^{-2+\frac{d}{p}}})} \sum_{n=0}^\infty \left(C\ve_0\right)^n
=\frac{1}{1-C\ve_0} \|b\|_{\wt{L^r_t(}\fB{_{p,\s}^{-2+\frac{d}{p}}})}.  
\end{aligned}
\end{equation}
It follows from a similar argument in \eqref{est;1/1+r} and \eqref{est;Banach_ring} that for all 
$1 \le r,\s \le \infty$, 
\begin{equation} \label{est;1/1+r2}
\left\|\frac{b}{1+b}\right\|_{\wt{L^r_t(}\fB{_{p,\s}^{\frac{d}{p}}})}
\les \|b\|_{\wt{L^r_t(}\fB{_{p,1}^{\frac{d}{p}}})}. 
\end{equation}

Lemma \ref{lem;bil2}, \eqref{est;Banach_ring}, \eqref{est;nonlin1}, \eqref{est;1/1+r} and \eqref{est;1/1+r2} gives us that 
\begin{equation} \label{est;nonlin2} 
\begin{aligned}
\|\div ((I(b)n \otimes n)\|_{\wt{L^1_t(}\fB{_{p,\s}^{-3+\frac{d}{p}}})}
\les& \left\|\frac{b}{1+b}\right\|_{\wt{L^\infty_t(}\fB{_{p,\s}^{-2+\frac{d}{p}}})}
      \|n\otimes n\|_{\wt{L^1_t(}\fB{_{p,\infty}^{\frac{d}{p}}})} \\
    &+\left\|\frac{b}{1+b}\right\|_{L^\infty_t(\fB{_{p,\infty}^{\frac{d}{p}}})}
      \|n\otimes n\|_{\wt{L^1_t(}\fB{_{p,\s}^{-2+\frac{d}{p}}})} \\
\les& 
\|b\|_{\wt{L^\infty_t(}\fB{_{p,\s}^{-2+\frac{d}{p}}})} \|n\|_{L^2_t(\fB{_{p,1}^{\frac{d}{p}}})}^2 \\
&+\|b\|_{L^\infty_t(\fB{_{p,1}^{\frac{d}{p}}})}
\|n\|_{\wt{L^2_t(}\fB{_{p,\s}^{-2+\frac{d}{p}}})}
\|n\|_{\wt{L^2_t(}\fB{_{p,1}^{\frac{d}{p}}})}.  
\end{aligned}
\end{equation}
Since it follows from the interpolation inequality that 
\begin{align*}
&\wt{L^2(I;}\fB{_{p,\s}^{-2+\frac{d}{p}}}) \hr 
\wt{L^\infty(I};\fB{_{p,\s}^{-3+\frac{d}{p}}}) \cap \wt{L^1(I;}\fB{_{p,\s}^{-1+\frac{d}{p}}}), \\
&\wt{L^2(I;}\fB{_{p,1}^{\frac{d}{p}}}) \hr 
\wt{L^\infty(I};\fB{_{p,1}^{-1+\frac{d}{p}}}) \cap L^1(I;\fB{_{p,1}^{1+\frac{d}{p}}}) 
\end{align*}
with $I=(0,t)$, 
we obtain by combining \eqref{est;nonlin1} and \eqref{est;nonlin2} that 
$$
\|\div ((I(b)-1)\,n \otimes n)\|_{\wt{L^1_t(}\fB{_{p,\s}^{-3+\frac{d}{p}}})} 
\les \|(b,n)\|_{CL^{(p,\s)}_T}^2+\|(b,n)\|_{CL^{(p,\s)}_T}^3.  
$$

\noindent
{\it The estimate for $I_P(b)\N b$}: 
Since $P$ is real analytic in a neighborhood of $1$ and $P'(1)=0$, 
there exists some constant $R_P>0$ such that 
if $|b| \le R_P$ then 
$$
  P(b+1)=P(1)+\sum_{n=2}^\infty a_n b^n \;\;\text{with}\;\;a_n=\frac{P^{(n)}(1)}{n!}. 
$$
Noting that $I_P(b)\N b=\N (b^2 \wt{I}_P(b))$ with $\wt{I}_P(b)=\sum_{n=2}^\infty a_n b^{n-2}$ 
and using \eqref{est;Banach_ring}, Lemma \ref{lem;bil2} and \eqref{est;1/1+r}, one can see that the following estimate for 
$\wt{I}_P(b)$ holds true:   
\begin{equation} \label{est;tlI_P}
\begin{aligned}
\|\wt{I}_P(b) b\|_{\wt{L^\infty_t(}\fB{_{p,\s}^{-2+\frac{d}{p}}})}
&=\left\|\sum_{n=2}^\infty a_n b^{n-1}\right\|_{\wt{L^\infty_t(}\fB{_{p,\s}^{-2+\frac{d}{p}}})} 
\les \bar{P}\left(C\|b\|_{\wt{L^\infty_t(}\cfB^\frac{d}{p})}\right) 
\|b\|_{\wt{L^\infty_t(}\fB{_{p,\s}^{-2+\frac{d}{p}}})}, 
\end{aligned}
\end{equation}
where $\bar{P}(z):=\sum_{n=2}^\infty |a_n|z^n$. If $\|b\|_{\wt{L^\infty_t(}\cfB^\frac{d}{p})} \le \frac{R_P}{2C}$, 
it follows that  
\begin{align} \label{est;I_P}
\|\wt{I}_P(b) b\|_{\wt{L^\infty_t(}\fB{_{p,\s}^{-2+\frac{d}{p}}})} 
\le D \|b\|_{\wt{L^\infty_t(}\fB{_{p,\s}^{-2+\frac{d}{p}}})} 
\;\;\text{with}\;\;D=1+\sup_{|z|\le\frac{R_P}{2}}|\bar{P}(z)|.
\end{align}
In a similar way to obtaining \eqref{est;I_P}, we also obtain by Lemma \ref{lem;prod_est} (1) that 
\begin{align} \label{est;I_P2}
\|\wt{I}_P(b)b\|_{\wt{L^1_t(}\fB{_{p,\s}^\frac{d}{p}})} \le D\|b\|_{\wt{L^1_t(}\fB{_{p,\s}^\frac{d}{p}})}.
\end{align}
By Lemma \ref{lem;bil2}, \eqref{est;I_P} and \eqref{est;I_P2}, we immediately obtain that 
\begin{align*}
\|\N(b^2 \wt{I}_P(b))\|_{\wt{L^1_t(}\fB{_{p,\s}^{-3+\frac{d}{p}}})}
\les& \|\wt{I}_P(b)b\|_{\wt{L^\infty_t(}\fB{_{p,\s}^{-2+\frac{d}{p}}})}
     \|b\|_{\wt{L^1_t(}\fB{_{p,\infty}^{\frac{d}{p}}})} \\
     &+\|\wt{I}_P(b)b\|_{\wt{L^1_t(}\fB{_{p,\infty}^{\frac{d}{p}}})}
      \|b\|_{\wt{L^\infty_t(}\fB{_{p,\s}^{-2+\frac{d}{p}}})} \\
\les& \|b\|_{\wt{L^\infty_t(}\fB{_{p,\s}^{-2+\frac{d}{p}}})} \|b\|_{\wt{L^1_t(}\fB_{p,\infty}^{\frac{d}{p}})}. 
\end{align*}
Therefore, we obtain by gathering all that  
$$
\|I_P(b) \N b\|_{\wt{L^1_t(}\fB{_{p,\s}^{-3+\frac{d}{p}}})}
\les \|(b,n)\|_{CL_T^{(p,\s)}}^2. 
$$
\noindent
{\it The estimate for $\mathcal{L}(I(b)n)$}: Since 
$\Delta (I(b)n)=\div (\N I(b)\otimes n+I(b)\N n)$, it holds that 
\begin{align*}
\|\Delta (I(b)n)\|_{\wt{L^1_t(}\fB{_{p,\s}^{-3+\frac{d}{p}}})}
\les& \|\N I(b)\otimes n\|_{\wt{L^1_t(}\fB{_{p,\s}^{-2+\frac{d}{p}}})}
+\|I(b) \N n\|_{\wt{L^1_t(}\fB{_{p,\s}^{-2+\frac{d}{p}}})}\\
\les& \|\N I(b)\|_{\wt{L^2_t(}\fB{_{p,\s}^{-2+\frac{d}{p}}})}
       \|n\|_{\wt{L^2_t(}\fB{_{p,\infty}^{\frac{d}{p}}})}
      +\|\N I(b)\|_{\wt{L^2_t(}\fB{_{p,\infty}^{\frac{d}{p}}})} 
        \|n\|_{\wt{L^2_t(}\fB{_{p,\s}^{-2+\frac{d}{p}}})} \\
     &+\|I(b)\|_{\wt{L^\infty_t(}\fB{_{p,\s}^{-2+\frac{d}{p}}})}
       \|\N n\|_{\wt{L^1_t(}\fB{_{p,\infty}^{\frac{d}{p}}})}
      +\|I(b)\|_{\wt{L^\infty_t(}\fB{_{p,\infty}^{\frac{d}{p}}})} 
        \|\N n\|_{\wt{L^1_t(}\fB{_{p,\s}^{-2+\frac{d}{p}}})}. 
\end{align*}
For $\N I(b)$, it follows from $\N I(b)=-(1+b)^{-2} \N b$ and \eqref{est;1/1+r} with $r=2$ that 
\begin{equation} \label{est;Na}
\begin{aligned}
\|\N I(b)\|_{\wt{L^2_t(}\fB{_{p,\s}^{-2+\frac{d}{p}}})}
= \left\|\frac{\N b}{(1+b)^2}\right\|_{\wt{L^2_t(}\fB{_{p,\s}^{-2+\frac{d}{p}}})} 
\les \left\|\frac{\N b}{1+b}\right\|_{\wt{L^2_t(}\fB{_{p,\s}^{-2+\frac{d}{p}}})}
\les \|\N b\|_{\wt{L^2_t(}\fB{_{p,\s}^{-2+\frac{d}{p}}})},  
\end{aligned}
\end{equation}
and similar estimate; 
\begin{equation} \label{est;Na2}
\begin{aligned}
\|\N I(b)\|_{\wt{L^2_t(}\fB{_{p,\infty}^{\frac{d}{p}}})}
\les \|\N b\|_{L^2_t(\fB_{p,1}^{\frac{d}{p}})}.   
\end{aligned}
\end{equation}
The above estimates \eqref{est;Na} and \eqref{est;Na2} gives us that 
\begin{align*}
\|\Delta (I(b)n)\|_{\wt{L^1_t(}\fB{_{p,\s}^{-3+\frac{d}{p}}})}
\les& \|b\|_{\wt{L^2_t(}\fB{_{p,\s}^{-1+\frac{d}{p}}})}
       \|n\|_{\wt{L^2_t(}\fB{_{p,1}^{\frac{d}{p}}})}
      +\|b\|_{L^2_t(\fB{_{p,1}^{1+\frac{d}{p}}})} 
        \|n\|_{\wt{L^2_t(}\fB_{p,\infty}^{-2+\frac{d}{p}})} \\
     &+\|b\|_{\wt{L^\infty_t(}\fB{_{p,\s}^{-2+\frac{d}{p}}})}
       \|n\|_{L^1_t(\fB{_{p,1}^{1+\frac{d}{p}}})}
      +\|b\|_{L^\infty_t(\fB{_{p,1}^{\frac{d}{p}}})} 
        \|n\|_{\wt{L^1_t(}\fB{_{p,\s}^{-1+\frac{d}{p}}})} \\
\les& \|(b,n)\|_{CL_T^{(p,\s)}}^2. 
\end{align*}
Performing the same calculation as the above estimate, one can obtain that 
\begin{align*}
\|\N \div (I(b)n)\|_{\wt{L^1_t(}\fB{_{p,\s}^{-3+\frac{d}{p}}})} \les \|(b,n)\|_{CL_T^{(p,\s)}}^2. 
\end{align*}

\noindent
{\it The estimate for $\div \mathcal{K}(b)$}: 
By the definition of $\mathcal{K}(b)$ in \eqref{eqn;Korteweg}, 
Lemma \ref{lem;bil2} and \eqref{est;1/1+r}, we have by noting 
$\Delta b^2=2b \Delta b+2|\N b|^2$ that 
\begin{align*}
\|\div \mathcal{K}(b)\|_{\wt{L^1_t(}\fB{_{p,\s}^{-3+\frac{d}{p}}})}
\les& \|\Delta b^2\|_{\wt{L^1_t(}\fB{_{p,\s}^{-2+\frac{d}{p}}})}
+\||\N b|^2\|_{\wt{L^1_t(}\fB{_{p,\s}^{-2+\frac{d}{p}}})}
      +\|\N b \otimes \N b\|_{\wt{L^1_t(}\fB{_{p,\s}^{-2+\frac{d}{p}}})} \\
\les& \|b\|_{\wt{L^\infty_t(}\fB{_{p,\s}^{-2+\frac{d}{p}}})} 
\|\Delta b\|_{\wt{L^1_t(}\fB{_{p,\infty}^{\frac{d}{p}}})}
+ \|b\|_{\wt{L^\infty_t(}\fB{_{p,\infty}^{\frac{d}{p}}})} 
\|\Delta b\|_{\wt{L^1_t(}\fB{_{p,\s}^{-2+\frac{d}{p}}})} \\
&+\|\N b\|_{\wt{L^2_t(}\fB{_{p,\s}^{-2+\frac{d}{p}}})}
\|\N b\|_{\wt{L^2_t(}\fB{_{p,\infty}^{\frac{d}{p}}})} \\
\les& \|b\|_{\wt{L^\infty_t(}\fB{_{p,\s}^{-2+\frac{d}{p}}})} 
\|b\|_{L^1_t(\fB{_{p,1}^{2+\frac{d}{p}}})}
+ \|b\|_{L^\infty_t(\fB{_{p,1}^{\frac{d}{p}}})} 
\|b\|_{\wt{L^1_t(}\fB{_{p,\s}^{\frac{d}{p}}})} \\
&+\|b\|_{\wt{L^2_t(}\fB{_{p,\s}^{-1+\frac{d}{p}}})}
\|b\|_{L^2_t(\fB{_{p,1}^{1+\frac{d}{p}}})} \\
\les& \|(b,n)\|_{CL_T^{(p,\s)}}^2. 
\end{align*}
Gathering all, we thus obtain the desired estimate \eqref{est;nonlin_all}. 
\end{proof}

\begin{proof}[The proof of Theorem \ref{thm;GWP_FLp}]  
The proof readily follows from Propositon \ref{prop;MR3}, Lemma \ref{lem;nonlin_est} and the following auxiliary lemma 
(for the proof, see e.g., \cite{Ch-Ko}):
\begin{lem}[\cite{Ch-Ko}] \label{lem;Banach_fixed_pt}
Let $(X,\|\cdot\|_X)$ be a Banach space. 
Let $B:X \times X \to X$ be a bilinear continuous operator with norm $K_2$ and 
$T:X \times X \times X \to X$ be a trilinear operators with norm $K_3$. 
Let further $L:X \to X$ be a continuous linear operator with norm $N<1$. 
Then for all $y \in X$ such that 
$$
\|y\|_X <\min \left(\frac{1-N}{2},\frac{(1-N)^2}{2(2K_2+3K_3)}\right), 
$$
the equation $x=y+L(x)+B(x,x)+T(x,x,x)$ has a unique solution $x$ in the ball $B_{\tilde{R}}^X(0)$ of center $0$ and 
radius $\tilde{R}=\min (1,\frac{1-N}{2K_2+3K_3})$. In addition, $x$ satisfies 
$$
\|x\|_X \le \frac{2}{1-N}\|y\|_X. 
$$
\end{lem}
By Proposition \ref{prop;MR3}, we see that $(a,m)$ satisfies 
\begin{equation} \label{est;pre_unif}
\begin{aligned}
\|(a,m)\|_{CL^{(p,\s)}_T} 
\les& \|(|\N|a_0,m_0)\|_{\fB{_{p,\s}^{-3+\frac{d}{p}}} \cap \fB{_{p,1}^{-1+\frac{d}{p}}}}  
      +\|\mathcal{N}(b,n)\|_{\wt{L^1_t(}\fB{_{p,\s}^{-3+\frac{d}{p}}}) \cap L^1_t(\fB{_{p,1}^{-1+\frac{d}{p}}})}.  
\end{aligned}
\end{equation}
Lemma \ref{lem;nonlin_est} ensures that $\mathcal{N}(b,n)$ can be regarded as a combination of 
bi-and-trilinear continuous operators in $CL^{(p,\s)}_T$ for any $T>0$. We define a ball in $CL^{(p,\s)}_T$ centered at the origin by 
$$
B_R^{CL^{(p,\s)}_T}(0):=\{(a,m) \in CL^{(p,\s)}_T;\|(a,m)\|_{CL^{(p,\s)}_T} \le R\}, 
$$ 
where $R>0$. Lemma \ref{lem;Banach_fixed_pt} shows that the existence of a unique solution in $B_R^{CL^{(p,\s)}_T}(0)$ for a sufficiently small data. 
Moreover, we are able to take $T=\infty$, thanks to the uniform estimate with respect to $t$ in Lemma \ref{lem;nonlin_est}. 
Therefore, we obtain a global solution satisfying \eqref{est;CL^p_T}. 
\end{proof}

\subsection{Analyticity of solutions} 
In this section, let us consider the proof of the analyticity for the global solution 
$u$ as constructed in the previous section. 
We now set $A=A(t,x):=e^{\sqrt{c_0t}|\N|}a(t,x)$ and $M=M(t,x):=e^{\sqrt{c_0t}|\N|}m(t,x)$. 
Here $c_0>0$ is the constant appeared in \eqref{est;ptw} of Lemma \ref{lem;pw_L}. 
Multiplying the both side of \eqref{est;ptw} by $\wh{\phi}_j e^{\sqrt{c_0 t}|\xi|}$, we have 
\begin{equation} \label{est;AN}
\begin{aligned}
\big|\big(|\xi|\wh{A}_j,\wh{M}_j\big)\big| 
\les& e^{\sqrt{c_0 t}|\xi|-tc_0|\xi|^2}\big|\wh{\phi}_j\big(|\xi|\wh{a}_0,\wh{m}_0\big)\big| 
+\int_0^t e^{\sqrt{c_0 t}|\xi|-(t-\t)c_0|\xi|^2}\big|\wh{\phi}_j\wh{\mathcal{N}}(a,m)\big|d\t \\
\les& e^{-\frac{c_0}{2}t|\xi|^2}\big|\wh{\phi}_j\big(|\xi|\wh{a}_0,\wh{m}_0\big)\big| 
+\int_0^t e^{-(t-\t)\frac{c_0}{2}|\xi|^2}\big|\wh{\phi}_j e^{\sqrt{c_0 \t}|\xi|}\wh{\mathcal{N}}(a,m)\big|d\t,
\end{aligned}
\end{equation}
where $(A_j,M_j):=(\Dj A, \Dj M)$ and  we have used the following uniform estimates (cf. Lemma 24.4 in \cite{Le}): 
$$
e^{\sqrt{c_0 t}|\xi|-\frac{c_0}{2} t|\xi|^2} \le \sqrt{e}, 
\quad e^{\sqrt{c_0 t}|\xi|-\sqrt{c_0 \t}|\xi|-\frac{c_0}{2}(t-\t) |\xi|^2} \le e^2. 
$$
Therefore we obtain by taking $L^{p'}_\xi$-norm of \eqref{est;AN} that 
\begin{equation} \label{est;AN2}
\begin{aligned}
\|(|\xi|\wh{A}_j,\wh{M}_j)\|_{L^{p'}_\xi}  
\les& e^{-\tilde{c}_0 t 2^{2j}}\|\wh{\phi}_j(|\xi|\wh{a}_0,\wh{m}_0)\|_{L^{p'}_\xi} \\
&+\int_0^t e^{-(t-\t)\tilde{c}_02^{2j}}
\|\wh{\phi}_j e^{\sqrt{c_0 \t}|\xi|}\wh{\mathcal{N}}(a,m)\|_{L^{p'}_\xi}d\t,
\end{aligned}
\end{equation}
where $\tilde{c}_0=c_0/8$. 
By the direct calculation, we immediately obtain 
\begin{equation} \label{est;FC}
\|e^{-\tilde{c}_0 2^{2j} \cdot}\|_{L^r(\R_+)} 
\les (2^{2j})^{-\frac{1}{r}} 
\end{equation}
and thus, we obtain by taking $L^r(\R_+)$-norm with respect to $t$ that 
\begin{equation}
\begin{aligned} \label{est;hom}
\|(|\xi|\wh{A}_j,\wh{M}_j)\|_{L^r(I;L^{p'}_\xi)}
\le& \|e^{-\tilde{c}_0 2^{2j} \cdot}\|_{L^r(\R_+)} \|\wh{\phi}_j(|\xi|\wh{a}_0,\wh{m}_0)\|_{L^{p'}_\xi} \\
&+ \|e^{-\tilde{c}_02^{2j}\cdot}\|_{L^\gm(\R_+)} 
   \|\wh{\phi}_j e^{\sqrt{c_0 t}|\xi|}\wh{\mathcal{N}}(a,m)\|_{L^{r_1}(I;L^{p'}_\xi)} \\
\les&  2^{-\frac{2}{r}j} \|\wh{\phi}_j(|\xi|\wh{a}_0,\wh{m}_0)\|_{L^{p'}_\xi} \\
&+2^{-\frac{2}{r}j}2^{(-2+\frac{2}{r_1})j}
 \|\wh{\phi}_j e^{\sqrt{c_0 t}|\xi|}\wh{\mathcal{N}}(a,m)\|_{L^{r_1}(I;L^{p'}_\xi)},   
\end{aligned}
\end{equation}
where $1 \le r_1 \le r$ and $\gm$ satisfies $\frac{1}{\gm}=1+\frac{1}{r}-\frac{1}{r_1}$. 
Multiplying the both sides by $2^{sj}$ and taking $\ell^\s$-norm, we obtain 
\begin{align} \label{est;GV}
\|(|\N|A,M)\|_{\wt{L^r(I;}\fB{_{p,\s}^{s+\frac{2}{r}}})} 
\les \|(|\N|a_0,m_0)\|_{\fB{_{p,\s}^s}}+\|e^{\sqrt{c_0 t}|\N|}\mathcal{N}(a,m)\|_{\wt{L^{r_1}(I};\fB{_{p,\s}^{s-2+\frac{2}{r_1}}})}. 
\end{align}
By the above estimate with $(r,r_1,s,\s)=(\infty,1,-3+\frac{d}{p},\s)$ and $(1,1,-3+\frac{d}{p},\s)$, 
we obtain that 
\begin{equation} \label{est;GV2}
\begin{aligned}
\|(|\N|A,M)\|_{L^\infty(I;\fB{_{p,\s}^{-3+\frac{d}{p}}}) \cap \wt{L^1(I;}\fB{_{p,\s}^{-1+\frac{d}{p}}}) } 
\les& \|(|\N|a_0,m_0)\|_{\fB{_{p,\s}^{-3+\frac{d}{p}}}} \\
&+\|e^{\sqrt{c_0 t}|\N|}\mathcal{N}(a,m)\|_{\wt{L^1(I;}\fB{_{p,\s}^{-3+\frac{d}{p}}})}.  
\end{aligned}
\end{equation}
Analogously, 
it follows from the above estimate with $(r,r_1,s,\s)=(\infty,1,-1+d/p,1)$ and 
$(1,1,-1+d/p,1)$ that 
\begin{equation} \label{est;GV3}
\begin{aligned}
\|(|\N|A,M)\|_{\wt{L^\infty(I};\fB{_{p,1}^{-1+\frac{d}{p}}}) \cap L^1(I;\fB{_{p,1}^{1+\frac{d}{p}}}) } 
\les& \|(|\N|a_0,m_0)\|_{\fB{_{p,1}^{-1+\frac{d}{p}}}} \\
&+\|e^{\sqrt{c_0 t}|\N|}\mathcal{N}(a,m)\|_{L^1(I;\fB{_{p,1}^{-1+\frac{d}{p}}})}. 
\end{aligned}
\end{equation}

Nextly, let us consider the estimation of the nonlinear terms. 
In order to complete the proof of Theorem \ref{thm;GWP_FLp}, we need to show that 
\begin{align} \label{est;nonlin_anal}
&\|e^{\sqrt{c_0 t}|\N|}\mathcal{N}(a,m)\|_{\wt{L^1(I;}\fB{_{p,\s}^{-3+\frac{d}{p}}})} 
\les \|(A,M)\|_{CL_T^{(p,\s)}}^2+\|(A,M)\|_{CL_T^{(p,\s)}}^3, \\
&\|e^{\sqrt{c_0 t}|\N|}\mathcal{N}(a,m)\|_{L^1(I;\fB{_{p,1}^{-1+\frac{d}{p}}})}
\les \|(A,M)\|_{CL_T^{(p,\s)}}^2+\|(A,M)\|_{CL_T^{(p,\s)}}^3. \label{est;nonlin_anal2}
\end{align}
In this paper, we only perform the estimate on $\div (I(a)m \otimes m)$ and $\N (a^2 \wt{I}_P(a))$ 
because the estimates on the other nonlinear terms are almost same as Lemma \ref{lem;nonlin_est}. 
Firstly, we give the key estimates as follows: 
\begin{lem}[{\it Bilinear estimates I\hspace{-0.4mm}I\hspace{-0.4mm}I}\hspace{1mm}] \label{lem;bil_anal}
Let $d \ge 3$, $1 \le p <d$, $1 \le \s \le \infty$ and $1 \le r,r_i \le \infty$ $(i=1,2,3,4)$ satisfy 
$\frac{1}{r}=\frac{1}{r_1}+\frac{1}{r_2}=\frac{1}{r_3}+\frac{1}{r_4}$. 
There exists some constant $C>0$ such that 
\begin{align*}
\|e^{\sqrt{c_0 t}|\N|}fg\|_{\wt{L^r_T(}\fB{_{p,\s}^{-2+\frac{d}{p}}})}
\le C\left(\|F\|_{\wt{L^{r_1}_T(}\fB{_{p,\s}^{-2+\frac{d}{p}}})} 
           \|G\|_{\wt{L^{r_2}_T(}\fB{_{p,\infty}^{\frac{d}{p}}})}
         +\|F\|_{\wt{L^{r_3}_T(}\fB{_{p,\infty}^{\frac{d}{p}}})} 
          \|G\|_{\wt{L^{r_4}_T(}\fB{_{p,\s}^{-2+\frac{d}{p}}})}\right),  
\end{align*}
where $F:=e^{\sqrt{c_0 t}|\N|}f$, $G:=e^{\sqrt{c_0 t}|\N|}g$. 
\end{lem}

\begin{lem}[{\it Bilinear estimates I\hspace{-0.4mm}V}\hspace{1mm}] \label{lem;bil_anal2}
Let $d \ge 1$, $s>0$, $1 \le p, \s \le \infty$ and $1 \le r,r_i \le \infty$ $(i=1,2,3,4)$ satisfying 
$\frac{1}{r}=\frac{1}{r_1}+\frac{1}{r_2}=\frac{1}{r_3}+\frac{1}{r_4}$. 
There exists some constant $C>0$ such that 
\begin{align*}
\|e^{\sqrt{c_0 t}|\N|}fg\|_{\wt{L^r_T(}\fB{_{p,\s}^{s}})} 
\le C\left(\|F\|_{\wt{L^{r_1}_T(}\fB{_{p,\s}^{s}})} 
           \|G\|_{L^{r_2}_T(\wh{L}^\infty)}
         +\|F\|_{L^{r_3}_T(\wh{L}^\infty)} 
          \|G\|_{\wt{L^{r_4}_T(}\fB{_{p,\s}^{s}})} \right). 
\end{align*}
\end{lem}
As for the proof of Lemmas \ref{lem;bil_anal} and \ref{lem;bil_anal2}, you can see in \S \ref{sect;appendix}. 

In this paper, we omit the proof of the following product law \eqref{est;BA_anal} because 
it is straightforward that \eqref{est;BA_anal} follows from Lemma \ref{lem;bil_anal2} in the case of $p<\infty$ and 
Lemmas \ref{lem;equi} with \ref{lem;B_t} in the case of $p=\infty$: 
\begin{cor}[{\it Product estimates I\hspace{-0.4mm}I\hspace{-0.4mm}I}\hspace{1mm}] \label{cor;prod_anal3}
Let $d \ge 1$, $1 \le p <\infty$ and $1 \le r,r_i \le \infty$ $(i=1,2)$ satisfying 
$\frac{1}{r}=\frac{1}{r_1}+\frac{1}{r_2}$. 
There exists some positive constant $C>0$ such that 
\begin{equation} \label{est;BA_anal}
\|e^{\sqrt{c_0 t}|\N|}fg\|_{\wt{L^r_T(}\fB_{p,1}^{\frac{d}{p}})}
\le C\|F\|_{\wt{L^{r_1}_T(}\fB_{p,1}^{\frac{d}{p}})} 
     \|G\|_{\wt{L^{r_2}_T(}\fB_{p,1}^{\frac{d}{p}})}. 
\end{equation}
\end{cor}

\noindent
{\it The estimate for $e^{\sqrt{c_0 t}|\N|}\div (I(a)\,m \otimes m)$}: 
It follows from Lemma \ref{lem;bil_anal} that 
\begin{align*}
\|e^{\sqrt{c_0 t}|\N|}\div (I(a)\,m \otimes m)\|_{\wt{L^1_T(}\fB{_{p,\s}^{-3+\frac{d}{p}}})} 
\les& \|e^{\sqrt{c_0 t}|\N|} (I(a)\,m \otimes m)\|_{\wt{L^1_T(}\fB{_{p,\s}^{-2+\frac{d}{p}}})} \\
\les& \|e^{\sqrt{c_0 t}|\N|} I(a)\|_{\wt{L^\infty_T(}\fB{_{p,\s}^{-2+\frac{d}{p}}})}
      \|e^{\sqrt{c_0 t}|\N|} (m \otimes m)\|_{\wt{L^1_T(}\fB{_{p,\infty}^{\frac{d}{p}}})} \\
    &+\|e^{\sqrt{c_0 t}|\N|} I(a)\|_{L^\infty_T(\fB{_{p,\infty}^{\frac{d}{p}}})}
      \|e^{\sqrt{c_0 t}|\N|} (m \otimes m)\|_{\wt{L^1_T(}\fB{_{p,\s}^{-2+\frac{d}{p}}})}. 
\end{align*}
Combining Lemmas \ref{lem;bil_anal} and \ref{lem;bil_anal2} with 
the similar arguments as obtaining \eqref{est;I_P} and \eqref{est;I_P2} gives us the following 
estimate on the conposition function: 
\begin{align*}
\|e^{\sqrt{c_0 t}|\N|} I(a)\|_{\wt{L^\infty_T(}\fB{_{p,\s}^{-2+\frac{d}{p}}})}
\les \|A\|_{\wt{L^\infty_T(}\fB{_{p,\s}^{-2+\frac{d}{p}}})}, \quad 
\|e^{\sqrt{c_0 t}|\N|} I(a)\|_{L^\infty_T(\fB{_{p,\infty}^{\frac{d}{p}}})}
\les \|A\|_{L^\infty_T(\fB{_{p,1}^{\frac{d}{p}}})}. 
\end{align*}
On the other hand, we obtain by Lemmas \ref{lem;bil_anal} and \ref{lem;bil_anal2} that 
\begin{align*} 
&
\begin{aligned}
\|e^{\sqrt{c_0 t}|\N|} (m \otimes m)\|_{\wt{L^1_T(}\fB{_{p,\infty}^{\frac{d}{p}}})}
\les \|e^{\sqrt{c_0 t}|\N|} (m \otimes m)\|_{L^1_T(\fB{_{p,1}^{\frac{d}{p}}})}
\les \|M\|_{\wt{L^2_T(}\fB{_{p,1}^{\frac{d}{p}}})}^2, 
\end{aligned} \\
&
\begin{aligned}
\|e^{\sqrt{c_0 t}|\N|} (m \otimes m)\|_{\wt{L^1_T(}\fB{_{p,\s}^{-2+\frac{d}{p}}})}
\les \|M\|_{\wt{L^2_T(}\fB{_{p,\s}^{-2+\frac{d}{p}}})} 
     \|M\|_{\wt{L^2_T(}\fB{_{p,1}^{\frac{d}{p}}})}. 
\end{aligned}
\end{align*}
Therefore, we get 
\begin{align*}
\|e^{\sqrt{c_0 t}|\N|}\div (I(a)\,m \otimes m)\|_{\wt{L^1_T(}\fB{_{p,\s}^{-3+\frac{d}{p}}})}
\les& \|A\|_{\wt{L^\infty_T(}\fB{_{p,\s}^{-2+\frac{d}{p}}})} \|M\|_{L^2_T(\fB{_{p,1}^{\frac{d}{p}}})}^2 \\
    &+\|A\|_{L^\infty_T(\fB{_{p,1}^{\frac{d}{p}}})} \|M\|_{\wt{L^2_T(}\fB{_{p,\s}^{-2+\frac{d}{p}}})} 
      \|M\|_{\wt{L^2_T(}\fB{_{p,1}^{\frac{d}{p}}})} \\
 \les& \|(A,M)\|_{CL_T^p}^3. 
\end{align*}

As for the estimate regarding to $L^1_T(\fB{_{p,1}^{-1+\frac{d}{p}}})$-norm, we see from 
Corollary \ref{cor;prod_anal3} that 
\begin{align*}
\|e^{\sqrt{c_0 t}|\N|}\div (I(a)\,m \otimes m)\|_{L^1_T(\fB{_{p,1}^{-1+\frac{d}{p}}})} 
\les& \|e^{\sqrt{c_0 t}|\N|}(I(a)\,m \otimes m)\|_{L^1_T(\fB{_{p,1}^{\frac{d}{p}}})} \\
\les& \|e^{\sqrt{c_0 t}|\N|}I(a)\|_{\wt{L^\infty_T(}\fB{_{p,1}^{\frac{d}{p}}})}
      \|e^{\sqrt{c_0 t}|\N|}(m \otimes m)\|_{L^1_T(\fB{_{p,1}^{\frac{d}{p}}})} \\
\les& \|A\|_{\wt{L^\infty_T(}\fB{_{p,1}^{\frac{d}{p}}})} 
      \|M\|_{\wt{L^2_T(}\fB{_{p,1}^\frac{d}{p}})}^2
\les \|(A,M)\|_{CL_T^p}^3, 
\end{align*}
where we used the following estimate which is given by a similar arguments as in \eqref{est;1/1+r} with 
Corollary \ref{cor;prod_anal3}: 
$$
\|e^{\sqrt{c_0 t}|\N|}I(a)\|_{\wt{L^\infty_T(}\fB{_{p,1}^{\frac{d}{p}}})} 
\les \|A\|_{\wt{L^\infty_T(}\fB{_{p,1}^{\frac{d}{p}}})}. 
$$
Gathering all, we arrive at 
\begin{equation} \label{est;nonlin_anal3}
\|e^{\sqrt{c_0 t}|\N|}\div (I(a)\,m \otimes m)\|
_{\wt{L^1_T(}\fB{_{p,\s}^{-3+\frac{d}{p}}}) \cap L^1_T(\fB{_{p,1}^{-1+\frac{d}{p}}})}
\les \|(A,M)\|_{CL_T^{(p,\s)}}^3. 
\end{equation}

\noindent
{\it The estimate for $e^{\sqrt{c_0 t}|\N|} \N (a^2 \wt{I}_P(a))$}: 
It follows from Lemma \ref{lem;bil_anal} that 
\begin{align*}
\|e^{\sqrt{c_0 t}|\N|} \N (a^2 \wt{I}_P(a))\|_{\wt{L^1_T(}\fB{_{p,\s}^{-3+\frac{d}{p}}})} 
\les& \|e^{\sqrt{c_0 t}|\N|} (a^2 \wt{I}_P(a))\|_{\wt{L^1_T(}\fB{_{p,\s}^{-2+\frac{d}{p}}})} \\
\les& \|A\|_{\wt{L^\infty_T(}\fB{_{p,\s}^{-2+\frac{d}{p}}})}
      \|e^{\sqrt{c_0 t}|\N|} (a \wt{I}_P(a))\|_{\wt{L^1_T(}\fB{_{p,\infty}^{\frac{d}{p}}})} \\
    &+\|A\|_{\wt{L^1_T(}\fB{_{p,\infty}^{\frac{d}{p}}})}
      \|e^{\sqrt{c_0 t}|\N|} (a \wt{I}_P(a))\|_{\wt{L^\infty_T(}\fB{_{p,\s}^{-2+\frac{d}{p}}})}. 
\end{align*}
Since we can obtain by similar argument as obtaining \eqref{est;I_P} and \eqref{est;I_P2} 
with Lemmas \ref{lem;bil_anal} and \ref{lem;bil_anal2} that 
\begin{align*}
\|e^{\sqrt{c_0 t}|\N|} (a \wt{I}_P(a))\|_{\wt{L^1_T(}\fB{_{p,\infty}^{\frac{d}{p}}})} 
\les \|A\|_{\wt{L^1_T(}\fB{_{p,\infty}^{\frac{d}{p}}})}, \quad 
\|e^{\sqrt{c_0 t}|\N|} (a \wt{I}_P(a))\|_{\wt{L^\infty_T(}\fB{_{p,\s}^{-2+\frac{d}{p}}})} \les \|A\|_{\wt{L^\infty_T(}\fB{_{p,\s}^{-2+\frac{d}{p}}})},  
\end{align*}
we thus obtain by gathering all that  
\begin{align*}
\|e^{\sqrt{c_0 t}|\N|} \N (a^2 \wt{I}_P(a))\|_{\wt{L^1_T(}\fB{_{p,\s}^{-3+\frac{d}{p}}})} 
\les& \|e^{\sqrt{c_0 t}|\N|} (a^2 \wt{I}_P(a))\|_{\wt{L^1_T(}\fB{_{p,\s}^{-2+\frac{d}{p}}})} \\
\les& \|A\|_{\wt{L^\infty_T(}\fB{_{p,\s}^{-2+\frac{d}{p}}})}
      \|A\|_{\wt{L^1_T(}\fB{_{p,\infty}^{\frac{d}{p}}})} 
\les \|(A,M)\|_{CL^{(p,\s)}_T}^2.   
\end{align*}

On the other hand, Lemma \ref{lem;Bern} and Corollary \ref{cor;prod_anal3} gives us that 
\begin{align*}
\|e^{\sqrt{c_0 t}|\N|}\N (a^2 \wt{I}_P(a)) \|_{L^1_T(\fB{_{p,1}^{-1+\frac{d}{p}}})} 
\les& \|e^{\sqrt{c_0 t}|\N|}(a^2 \wt{I}_P(a)) \|_{L^1_T(\fB{_{p,1}^{\frac{d}{p}}})} \\
\les& \|A\|_{\wt{L^2_T(}\fB{_{p,1}^\frac{d}{p}})} 
      \|e^{\sqrt{c_0 t}|\N|}(a\wt{I}_P(a))\|_{\wt{L^2_T(}\fB{_{p,1}^\frac{d}{p}})} \\
\les& \|A\|_{\wt{L^2_T(}\fB{_{p,1}^\frac{d}{p}})}^2. 
\end{align*} 
Since it holds true that for some $j_0 \in \Z$, 
\begin{align*}
\|A\|_{\wt{L^2_T(}\fB{_{p,1}^\frac{d}{p}})} 
=& \sum_{j \le j_0-1} 2^{\frac{d}{p}j}\|\Dj A\|_{L^2(I;\wh{L}^p)} 
  + \sum_{j \ge j_0} 2^{\frac{d}{p}j}\|\Dj A\|_{L^2(I;\wh{L}^p)} \\
\les& \|A\|_{\wt{L^2_T(}\fB{_{p,\infty}^{-1+\frac{d}{p}}})} \sum_{j \le j_0-1} 2^{j}
      + \|A\|_{\wt{L^2_T(}\fB{_{p,1}^{1+\frac{d}{p}}})} \sup_{j \ge j_0} 2^{-j} 
\les \|(A,M)\|_{CL_T^{(p,\s)}}^2, 
\end{align*}
we thus obtain by gathering all that 
\begin{equation} \label{est;nonlin_anal4}
\|e^{\sqrt{c_0 t}|\N|} \N (a^2 \wt{I}_P(a))\|
_{\wt{L^1_T(}\fB{_{p,\s}^{-3+\frac{d}{p}}}) \cap L^1_T(\fB{_{p,1}^{-1+\frac{d}{p}}})}
\les \|(A,M)\|_{CL_T^{(p,\s)}}^2. 
\end{equation}

By similar arguments, one can obtain the desired estimates \eqref{est;nonlin_anal} and \eqref{est;nonlin_anal2}. 
Therefore, gathering all, we arrive at 
\begin{align*}
\|(A,M)\|_{CL_T^{(p,\s)}} \les \|(|\N|a_0,m_0)\|_{\fB{_{p,\s}^{-3+\frac{d}{p}}} \cap \fB{_{p,1}^{-1+\frac{d}{p}}}}+ \|(A,M)\|_{CL_T^{(p,\s)}}^2+\|(A,M)\|_{CL_T^{(p,\s)}}^3. 
\end{align*}
Since the smallness assumption $\|(|\N|a_0,m_0)\|_{\fB{_{p,\s}^{-3+\frac{d}{p}}} \cap \fB{_{p,1}^{-1+\frac{d}{p}}}} \ll 1$, we obtain the uniform estimate of $(A,M)$ 
in $CL_T^{(p,\s)}$. We thus complete the proof of Theorem \ref{thm;GWP_FLp}.

\sect{Time Decay Estimates} \label{sect;Decay}


\begin{prop} \label{prop;negative} 
Let $1 \le p \le 2$ and suppose that the initial data $(a_0,m_0)$ satisfies 
the same assumption as in Theorem \ref{thm;Lp-L1}. There exists some positive constant $C>0$ such that 
global-in-time solution $(a,m)$ obtained by Theorem \ref{thm;GWP_FLp} satisfies
\begin{align*}
\|(|\N|A, M)\|
_{L^\infty(I;\fB{_{p,\infty}^{-1-\frac{d}{p'}}}) \cap \wt{L^1(I;}\fB_{p,\infty}^{1-\frac{d}{p'}})}
\le C \|(a_0,\tilde{m}_0)\|_{\fB{_{p,\infty}^{-\frac{d}{p'}}}}. 
\end{align*}

\end{prop}

Lemmas \ref{lem;diagonal3} and \ref{lem;offd3} give us the following Lemma: 
\begin{lem}[{\it Product estimates I\hspace{-0.4mm}V}\hspace{1mm}]  \label{lem;prod_anal4}
Let $s_i \in \R$, $1 \le p,p_i,r,r_i \le \infty$ $(i=1,2)$ satisfying 
\begin{gather*}
\frac{1}{p}\le\frac{1}{p_1}+\frac{1}{p_2}, \quad s_1 \le \frac{d}{p_1}, \quad s_2 <\frac{d}{p_2}, \quad 
s_1+s_2+d\min\left(0,1-\frac{1}{p_1}-\frac{1}{p_2}\right) \ge 0, \\
\frac{1}{r}=\frac{1}{r_1}+\frac{1}{r_2}. 
\end{gather*}
Then there exists some constant $C>0$ such that 
\begin{align} \label{est;prod_diag}
\left\|e^{\sqrt{c_0 t}|\N|} fg\right\|_{\wt{L^r_T(}\fB{_{p,\infty}^{s_1+s_2+\frac{d}{p}-\frac{d}{p_1}-\frac{d}{p_2}}})}
\le C\|F\|_{\wt{L^{r_1}_T(}\fB{_{p_1,1}^{s_1}})}\|G\|_{\wt{L^{r_2}_T(}\fB{_{p_2,\infty}^{s_2}})}. 
\end{align}
\end{lem}

\begin{cor}
In the case $1 \le p \le 2$, the following estimates are obtained from Lemma \ref{lem;prod_anal4} with $d \ge 1$:   
\begin{align}
&\|e^{\sqrt{c_0 t}|\N|} fg\|_{\wt{L^r_T(}\fB_{p,\infty}^{-\frac{d}{p'}})}
\le C \|F\|_{\wt{L^{r_1}_T(}\fB{_{p,1}^\frac{d}{p}})}
\|G\|_{\wt{L^{r_2}_T(}\fB{^{-\frac{d}{p'}}_{p,\infty}})}. \label{est;p-1} 
\end{align}
\end{cor}

\begin{proof}[The proof of Poroposion \ref{prop;negative}]
First of all, we see from the estimate \eqref{est;GV} that 
\begin{equation} \label{est;GV_neg}
\begin{aligned}
\|(|\N|A,M)\|_{L^\infty(I;\fB{_{p,\infty}^{-1-\frac{d}{p'}}}) \cap \wt{L^1(I;}\fB{_{p,\infty}^{1-\frac{d}{p'}}}) } 
\les& \|(|\N|a_0,m_0)\|_{\fB{_{p,\infty}^{-1-\frac{d}{p'}}}} \\
&+\|e^{\sqrt{c_0 t}|\N|}\mathcal{N}(a,m)\|_{\wt{L^1(I;}\fB{_{p,\infty}^{-1-\frac{d}{p'}}})}. 
\end{aligned}
\end{equation}

In this paper, we give the estimate for the nonlinear terms 
$\div (I(a) m \otimes m)$ and $\N (a^2 \wt{I}_P(a))$ only. 
The product law \eqref{est;p-1} gives us that 
\begin{align*}
\|e^{\sqrt{c_0 t}|\N|} \div (I(a) m \otimes m)\|_{\wt{L^1_T(}\fB{_{p,\infty}^{-1-\frac{d}{p'}}})}
\les& \|e^{\sqrt{c_0 t}|\N|}I(a)\|_{\wt{L^\infty_T(}\fB{_{p,1}^{\frac{d}{p}}})}
     \|e^{\sqrt{c_0 t}|\N|}m \otimes m\|_{\wt{L^1_T(}\fB{_{p,\infty}^{-\frac{d}{p'}}})} \\
\les& \|A\|_{\wt{L^\infty_T(}\fB{_{p,1}^{\frac{d}{p}}})}
      \|M\|_{\wt{L^2_T(}\fB{_{p,\infty}^{-\frac{d}{p'}}})}
      \|M\|_{\wt{L^2_T(}\fB{_{p,1}^{\frac{d}{p}}})} \\
\les& \|(A,M)\|_{CL_T^{(p,1)}}^2 \|M\|_{\wt{L^2_T(}\fB{_{p,\infty}^{-\frac{d}{p'}}})}. 
\end{align*}
Analogously, we obtain from \eqref{est;p-1} that 
\begin{align*}
\|e^{\sqrt{c_0 t}|\N|}\N(a^2 \wt{I}_P(a))\|_{\wt{L^1_T(}\fB{_{p,\infty}^{-1-\frac{d}{p'}}})}
\les \|e^{\sqrt{c_0 t}|\N|} a \wt{I}_P(a)\|_{\wt{L^\infty_T(}\fB{_{p,\infty}^{-\frac{d}{p'}}})}
     \|A\|_{\wt{L^1_T(}\fB{_{p,1}^{\frac{d}{p}}})}. 
\end{align*}
Since \eqref{est;p-1} holds true, we easily see that 
\begin{align*}
\|e^{\sqrt{c_0 t}|\N|} a \wt{I}_P(a)\|_{\wt{L^\infty_T(}\fB{_{p,\infty}^{-\frac{d}{p'}}})} 
\les \|A\|_{\wt{L^\infty_T(}\fB{_{p,\infty}^{-\frac{d}{p'}}})} 
\end{align*}
and thus, 
$$
\|e^{\sqrt{c_0 t}|\N|}\N (a^2 \wt{I}_P(a))\|_{\wt{L^1_T(}\fB{_{p,\infty}^{-1-\frac{d}{p'}}})}
\les \|(A,M)\|_{CL_T^{(p,1)}} \|A\|_{\wt{L^\infty_T(}\fB{_{p,\infty}^{-\frac{d}{p'}}})}. 
$$

In a similar way to the above arguments, we are able to obtain that 
\begin{align*}
&\|e^{\sqrt{c_0 t}|\N|}\mathcal{L}(I(a) m)\|_{\wt{L^1_T(}\fB{_{p,\infty}^{-1-\frac{d}{p'}}})} 
\les \|(A,M)\|_{CL_T^{(p,1)}} \|M\|_{\wt{L^2_T(}\fB{_{p,\infty}^{-\frac{d}{p'}}}) \cap \wt{L^1_T(}\fB{_{p,\infty}^{1-\frac{d}{p'}}})}, \\
&\|e^{\sqrt{c_0 t}|\N|}\div \mathcal{K}(a)\|_{\wt{L^1_T(}\fB{_{p,\infty}^{-1-\frac{d}{p'}}})} 
\les \|(A,M)\|_{CL_T^{(p,1)}} \|A\|_{\wt{L^\infty_T(}\fB{_{p,\infty}^{-\frac{d}{p'}}}) \cap \wt{L^2_T(}\fB{_{p,\infty}^{1-\frac{d}{p'}}})}, 
\end{align*}
Therefore, we obtain by gathering the above estimates and $m_0=\N \tilde{m}_0$ that 
\begin{align*}
\|(|\N|A,M)\|_{L^\infty(I;\fB{_{p,\infty}^{-1-\frac{d}{p'}}}) \cap \wt{L^1(I;}\fB{_{p,\infty}^{1-\frac{d}{p'}}}) } 
\les \|(a_0,\tilde{m}_0)\|_{\fB{_{p,\infty}^{-\frac{d}{p'}}}}.  
\end{align*}
This is our desired estimate. 
\end{proof}

\begin{proof}[The proof of Theorem \ref{thm;Lp-L1}] 
Thanks to Proposition \ref{prop;negative}, we obtain that for all $s>-1-d/p'$, 
\begin{align*}
t^{\frac{d}{2p'}+\frac{s+1}{2}}\|(|\N|a,m)(t)\|_{\fB{_{p,1}^s}} 
=&t^{\frac{d}{2p'}+\frac{s+1}{2}} 
  \sum_{j \in \Z} 2^{sj}\|e^{-\sqrt{c_0 t}|\cdot|} e^{\sqrt{c_0 t}|\cdot|} \wh{\phi}_j(|\xi|\wh{a},\wh{m})\|_{L_\xi^{p'}} \\
\les& \|(|\N|A,M)(t)\|_{\fB{_{p,\infty}^{-1-\frac{d}{p'}}}}
      \sum_{j \in \Z} (\sqrt{t}2^{j})^{\frac{d}{p'}+s+1} e^{-\frac{1}{4} \sqrt{c_0 t}2^j} \\
\les& \|(a_0,\tilde{m}_0)\|_{\fB{_{p,\infty}^{-\frac{d}{p'}}}}. 
\end{align*}
Here we used the convergence of the summation; 
$\sup_{t > 0}\sum_{j \in \Z} 
       (\sqrt{t}2^j)^{\alpha} e^{-c\sqrt{t}2^j}$ for all $c,\al>0$ 
(for the proof of this fact is completely same as the one of Lemma \ref{lem;cineq}).  
Therefore, we obtain the desired decay properties 
$\|(|\N|a,m)(t)\|_{\fB{_{p,1}^s}}=O(t^{-\frac{d}{2p'}-\frac{s+1}{2}})$ $(t \to \infty)$. 
\end{proof}

\sect{Asymptotic behavior of solutions} \label{sect;AP}

\subsection{Linear estimates} 

\begin{prop} \label{prop;AP1}
Let the initial data $(a_0,m_0)$ satisfy $m_0=\N \tilde{m}_0$ with $(a_0,\tilde{m}_0) \in L^1(\R^d)$ and $1< p \le \infty$.  
Then it holds that for all $s>-d/p'$, 
\begin{align}
&\lim_{t \to \infty} t^{\frac{d}{2}(1-\frac{1}{p})+\frac{s}{2}}
\left\||\N|^s \left(G^{1,1}(t,\cdot)*a_0-G_1(t) \int_{\R^d}a_0(y) dy\right)\right\|_{\wh{L}^p}=0, \label{est;G_1}\\
&\lim_{t \to \infty} t^{\frac{d}{2}(1-\frac{1}{p})+\frac{s}{2}}
\left\||\N|^s \left(G^{1,2}(t,\cdot)*m_0-G_2(t) \int_{\R^d}\tilde{m}_0(y) dy\right)\right\|_{\wh{L}^p}=0.  \label{est;G_2}
\end{align}
\end{prop}

\begin{proof}[The proof of Theorem \ref{prop;AP1}] 
We only prove \eqref{est;G_2} because the proof of \eqref{est;G_1} and \eqref{est;G_2} are almost same. 
First, we would like to treat the case of $\nu^2 \neq 4 \kappa$.  
Since $m_0=\N \tilde{m}_0$, it holds that 
\begin{align*}
\mathcal{F}[G^{1,2}(t,\cdot)*m_0]
=&-c_d \frac{e^{\lam_+(\xi)t}-e^{\lam_-(\xi)t}}{\lam_+-\lam_-}i \xi \cdot ( i\xi \wh{\tilde{m}}_0) \\
=& c_d \frac{e^{\lam_+(\xi)t}-e^{\lam_-(\xi)t}}{\sqrt{\nu^2-4 \kappa}} \wh{\tilde{m}}_0 
= c_d \left(\wh{G}_+(t)-\wh{G}_-(t)\right)\wh{\tilde{m}}_0, 
\end{align*}
where $c_d=(2 \pi)^{d/2}$. Remembering the definition of $G_2(t)$ under $\nu^2 \neq 4 \kappa$ 
(see \S \ref{subsect;Thm2.3}), we obtain 
\begin{align*}
\mathcal{F}\left[\left(G^{1,2}(t,\cdot)*m_0-G_2(t) \int_{\R^d}\tilde{m}_0(y) dy\right)\right]
=c_d \left(\wh{G}_+(t)-\wh{G}_-(t)\right) \cdot (\wh{\tilde{m}}_0(\xi)-\wh{\tilde{m}}_0(0)). 
\end{align*}
Since there exists some constant $c_0>0$ depending only on $\nu$ and $\kappa$ such that 
$|e^{\lam_\pm(\xi)t}| \les e^{-c_0 |\xi|^2 t}$, it is straightforward that we obtain 
\begin{align*}
\left\||\xi|^s \mathcal{F}\left[\left(G^{1,2}(t,\cdot)*m_0-G_2(t) \int_{\R^d}m_0(y) dy\right)\right]\right\|_{L^{p'}_\xi} 
\les 
\||\xi|^s e^{-c_0|\xi|^2t}(\wh{\tilde{m}}_0(\xi)-\wh{\tilde{m}}_0(0))\|_{L^{p'}_\xi}. 
\end{align*}
We easily obtain by performing $\xi \mapsto \eta/\sqrt{t}$ that 
\begin{align*}
\||\xi|^s e^{-c_0|\xi|^2t}(\wh{\tilde{m}}_0(\xi)-\wh{\tilde{m}}_0(0))\|_{L^{p'}_\xi}
=t^{-\frac{d}{2p'}-\frac{s}{2}}
\left\||\eta|^s e^{-c_0 |\eta|^2} \left(\wh{\tilde{m}}_0\left(\frac{\eta}{\sqrt{t}}\right)-\wh{\tilde{m}}_0(0)\right)\right\|_{L^{p'}_\eta}.  
\end{align*} 
Thanks to Riemann-Lebesgue's theorem and the boundedness $\||\eta|^s e^{-c_0 |\eta|^2}\|_{L^{p'}_\eta} < +\infty$ for all $s>-d/p'$, 
we obtain the desired estimate \eqref{est;G_2} with $\nu^2 \neq 4 \kappa$. 

In the case of $\nu^2=4\kappa$, $G^{1,2}(t,x)$ is given by 
$G^{1,2}(t,x)=\mathcal{F}^{-1}[-t e^{-\frac{\nu}{2}|\xi|^2 t} i {}^t\xi]$ (see \eqref{eqn;neq2}). 
Since $m_0=\N \tilde{m}_0$, it holds that 
\begin{align*}
\mathcal{F}[G^{1,2}(t,\cdot)*m_0]
=-c_d t e^{-\frac{\nu}{2}|\xi|^2 t} i \xi \cdot ( i\xi \wh{\tilde{m}}_0)
=c_d t |\xi|^2 e^{-\frac{\nu}{2}|\xi|^2 t} \wh{\tilde{m}}_0
=c_d \wh{G}_2(t) \wh{\tilde{m}}_0
\end{align*}
and thus, 
\begin{align*}
\mathcal{F}\left[\left(G^{1,2}(t,\cdot)*m_0-G_2(t) \int_{\R^d}m_0(y) dy\right)\right]
= c_d \wh{G}_2(t) \cdot (\wh{\tilde{m}}_0(\xi)-\wh{\tilde{m}}_0(0)). 
\end{align*}
Since $|\wh{G}_2(t)|=|t |\xi|^2e^{-\frac{\nu}{2}|\xi|^2 t}| \les e^{-\frac{\nu}{4} |\xi|^2 t}$, we easily obtain that 
\begin{align*}
\left\||\xi|^s \mathcal{F}\left[\left(G^{1,2}(t,\cdot)*m_0-G_2(t) \int_{\R^d}m_0(y) dy\right)\right]\right\|_{L^{p'}_\xi} 
\les 
\||\xi|^s e^{-\frac{\nu}{4}|\xi|^2t}(\wh{\tilde{m}}_0(\xi)-\wh{\tilde{m}}_0(0))\|_{L^{p'}_\xi}. 
\end{align*}
In a similar argument as in the case of $\nu^2 \neq 4 \kappa$, 
we obtain the desired estimate \eqref{est;G_2} with $\nu^2 = 4 \kappa$. 
Gathering all, we thus complete the proof of Proposition \ref{prop;AP1}. 
\end{proof}

Analogously, one can obtain the following estimates: 
\begin{prop} \label{prop;AP2}
Let $(a_0,\tilde{m}_0)$ satisfy the same assumption as in Proposition \ref{prop;AP1} and $1 < p \le \infty$.  
Then it holds that for all $s>-1-d/p'$, 
\begin{align}
&\lim_{t \to \infty} t^{\frac{d}{2}(1-\frac{1}{p})+\frac{s+1}{2}}
\left\||\N|^s \left(G^{2,1}(t,\cdot)*a_0-\N G_2(t) \int_{\R^d}a_0(y) dy\right)\right\|_{\wh{L}^p}=0, \label{est;NG_2}\\
&\lim_{t \to \infty} t^{\frac{d}{2}(1-\frac{1}{p})+\frac{s+1}{2}}
\left\||\N|^s \left(G^{2,2}(t,\cdot)*m_0-\N G_3(t) \int_{\R^d}\tilde{m}_0(y) dy\right)\right\|_{\wh{L}^p}=0.  \label{est;G_3}
\end{align}
\end{prop}

\subsection{The estimates on the $\widehat{L}^1$-norm}  

\begin{prop}[{\it $\wh{L}^1$-Uniform estimate}] \label{prop;FL^1}
Suppose the same assumption as in Theorem \ref{thm;Asympt}.  
Then the global-in-time solution $(a,m)$ obtained by Theorem \ref{thm;GWP_FLp} satisfies the following estimate: 
$$
\|(|\N|A,M)\|_{L^\infty(I;\fB{_{1,\infty}^{-1}}) \cap \wt{L^1(I};\fB{_{1,\infty}^1})} 
\le C\|(a_0,\tilde{m}_0)\|_{\fB{_{1,\infty}^0}}.  
$$
\end{prop}

\begin{lem}[{\it Bilinear estimates} V\hspace{0mm}] \label{lem;bil5}
Let $1 \le r_i \le \infty$ $(i=1,2,3,4)$ satisfying 
$$
\frac{1}{r}=\frac{1}{r_1}+\frac{1}{r_2}=\frac{1}{r_3}+\frac{1}{r_4}. 
$$ 
There exists some constant $C > 0$ such that the following estimate holds: 
\begin{equation} \label{est;prod_bl}
\|e^{\sqrt{c_0 t}|\N|}fg\|_{\wt{L^r_T(}\fB{_{1,\infty}^0})}
\le 
C\left(\|F\|_{\wt{L^{r_1}_T(}\fB{_{1,\infty}^0})} \|G\|_{L^{r_2}_T(\fB{^0_{\infty,1}})} 
+ \|F\|_{L^{r_3}_T(\fB{^0_{\infty,1}})}
\|G\|_{\wt{L^{r_4}_T(}\fB{^0_{1,\infty}})}\right). 
\end{equation}
\end{lem}   

\begin{proof}[The proof of Theorem \ref{thm;Asympt}] 
By using the estimate \eqref{est;GV}, we easily see that 
\begin{equation} \label{est;GV_L1hat}
\begin{aligned}
\|(|\N|A,M)\|_{L^\infty(I;\fB{_{1,\infty}^{-1}}) \cap \wt{L^1(I;}\fB{_{1,\infty}^1}) } 
\les& \|(|\N|a_0,m_0)\|_{\fB{_{1,\infty}^{-1}}} \\
&+\|e^{\sqrt{c_0 t}|\N|}\mathcal{N}(a,m)\|_{\wt{L^1(I;}\fB{_{1,\infty}^{-1}})}. 
\end{aligned}
\end{equation}
It follows from Lemma \ref{lem;Bern} and Lemma \ref{lem;bil5} that 
\begin{align*}
\|e^{\sqrt{c_0 t}|\N|} \div (I(a) m \otimes m)\|_{\wt{L^1_T(}\fB{_{1,\infty}^{-1}})}
\les& \|e^{\sqrt{c_0 t}|\N|} I(a) m \otimes m\|_{\wt{L^1_T(}\fB{_{1,\infty}^{0}})} \\
\les& \|e^{\sqrt{c_0 t}|\N|} I(a)\|_{L^\infty_T(\fB{_{1,\infty}^0})} \|e^{\sqrt{c_0 t}|\N|} m \otimes m\|_{L^1_T(\fB{_{\infty,1}^0})} \\
&+ \|e^{\sqrt{c_0 t}|\N|} I(a)\|_{L^\infty_T(\fB{_{\infty,1}^0})} \|e^{\sqrt{c_0 t}|\N|} m \otimes m\|_{\wt{L^1_T(}\fB{_{1,\infty}^0})}. 
\end{align*}
Since Lemma \ref{lem;sm} gives us that 
$\fB{_{p,1}^\frac{d}{p}}(\R^d) \hr \wh{\dot{B}}{_{\infty,1}^0}(\R^d)$ for all $1 \le p \le \infty$, 
it follows from Corollary \ref{cor;prod_anal3} and H\"older inequality that 
\begin{equation*}
\|e^{\sqrt{c_0 t}|\N|} m \otimes m\|_{L^1_T(\fB{_{\infty,1}^0})} 
\les \|M\|_{\wt{L^2_T(}\fB{_{p,1}^\frac{d}{p}})}^2, \quad 
\|e^{\sqrt{c_0 t}|\N|} I(a)\|_{L^\infty_T(\fB{_{\infty,1}^0})} 
\les \|A\|_{\wt{L^\infty_T(}\fB{_{p,1}^\frac{d}{p}})}. 
\end{equation*}
On the other hand, we obtain by Lemma \ref{lem;bil5} that 
\begin{equation*}
 \|e^{\sqrt{c_0 t}|\N|} I(a)\|_{L^\infty_T(\fB{_{1,\infty}^0})} 
 \les \|A\|_{L^\infty_T(\fB{_{1,\infty}^0})}, \;\;
 \|e^{\sqrt{c_0 t}|\N|} m \otimes m\|_{\wt{L^1_T(}\fB{_{1,\infty}^0})} 
 \les \|M\|_{\wt{L^2_T(}\fB{_{1,\infty}^0})} \|M\|_{L^2_T(\fB{_{p,1}^\frac{d}{p}})}. 
 \end{equation*}
 Gathering all, we then obtain that 
 \begin{equation*}
 \|e^{\sqrt{c_0 t}|\N|} \div (I(a) m \otimes m)\|_{\wt{L^1_T(}\fB{_{1,\infty}^{-1}})}
 \les \|(A,M)\|_{CL_T^{(p,1)}}^2 \|(|\N|A,M)\|_{\wt{L^\infty_T(}\fB{_{1,\infty}^{-1}}) \cap \wt{L^1_T(}\fB{_{1,\infty}^1})}. 
 \end{equation*}
 
It follows from Lemma \ref{lem;Bern} and Lemma \ref{lem;bil5} that 
\begin{align*}
\|e^{\sqrt{c_0 t}|\N|} \N (a^2 \wt{I}_P(a))\|_{\wt{L^1_T(}\fB{_{1,\infty}^{-1}})}
\les& \|A\|_{L^\infty_T(\fB{_{1,\infty}^0})} \|e^{\sqrt{c_0 t}|\N|} a \wt{I}_P(a)\|_{L^1_T(\fB{_{\infty,1}^0})} \\
&+ \|A\|_{L^1_T(\fB{_{\infty,1}^0})} \|e^{\sqrt{c_0 t}|\N|} a \wt{I}_P(a)\|_{L^\infty_T(\fB{_{1,\infty}^0})}. 
\end{align*}
Since it follows from Corollary \ref{cor;prod_anal3} and Lemma \ref{lem;bil5} that 
\begin{equation*}
\|e^{\sqrt{c_0 t}|\N|} a \wt{I}_P(a)\|_{L^1_T(\fB{_{\infty,1}^0})} 
\les \|A\|_{L^1_T(\fB{_{\infty,1}^0})}, \quad 
\|e^{\sqrt{c_0 t}|\N|} a \wt{I}_P(a)\|_{L^\infty_T(\fB{_{1,\infty}^0})}
\les \|A\|_{L^\infty_T(\fB{_{1,\infty}^0})}. 
\end{equation*}
Therefore, we obtain by gathering the above estimates and the embeddings 
$ \fB{_{p,1}^\frac{d}{p}}(\R^d) \hr \wh{\dot{B}}{_{\infty,1}^0}(\R^d)$ for all $1 \le p \le \infty$ that 
$$
\|e^{\sqrt{c_0 t}|\N|} \N (a^2 \wt{I}_P(a))\|_{\wt{L^1_T(}\fB{_{1,\infty}^{-1}})}
\les \|(A,M)\|_{CL_T^{(p,1)}} \|(|\N|A,M)\|_{\wt{L^\infty_T(}\fB{_{1,\infty}^{-1}}) \cap \wt{L^1_T(}\fB{_{1,\infty}^1})}. 
$$

By similar procedure, we easily obtain 
\begin{align*}
&\|e^{\sqrt{c_0 t}|\N|} \mathcal{L}(I(a)m)\|_{\wt{L^1_T(}\fB{_{1,\infty}^{-1}})}
+\|e^{\sqrt{c_0 t}|\N|} \div \mathcal{K}(a)\|_{\wt{L^1_T(}\fB{_{1,\infty}^{-1}})} \\
&\les \|(A,M)\|_{CL_T^{(p,1)}} \|(|\N|A,M)\|_{\wt{L^\infty_T(}\fB{_{1,\infty}^{-1}}) \cap \wt{L^1_T(}\fB{_{1,\infty}^1})}. 
\end{align*}
Combining the above estimate and Proposition \ref{prop;negative}, we have the desired estimate. 
\end{proof}

\begin{rem}
Thanks to the conculusion of Lemma \ref{prop;FL^1}, 
we are able to obtain that 
\begin{equation} \label{est;L1hat}
\|(|\N|a,m)\|_{L^\infty(\R_+;\fB{_{1,\infty}^{-1}}) \cap \wt{L^1(\R}{}_+;\fB{_{1,\infty}^1})} 
\le C\|(a_0,\tilde{m}_0)\|_{\fB{_{1,\infty}^0}}.  
\end{equation}
This is immediately obtained by Lemma \ref{prop;FL^1} and the following estimate: 
\begin{align*}
\|\Dj u\|_{L^r(I;\wh{L}^p)} 
= \|e^{-\sqrt{c_0 t}|\xi|} e^{\sqrt{c_0 t}|\xi|} \wh{\Dj u}\|_{L^r(I;\wh{L}^p)} 
\les \|\Dj U\|_{L^r(I;\wh{L}^p)}, 
\end{align*}
for some $u \in \mathcal{S}'$. Here $U:=e^{\sqrt{c_0 t}|\N|} u$. 
\end{rem}

The following statement can be obtained by gathering the same arguments as the proof of Theorem \ref{thm;Lp-L1} and 
the uniform estimate in Proposition \ref{prop;FL^1}: 
\begin{cor} 
Suppose that the same assumptions as in Proposition \ref{prop;FL^1}. 
Then it holds true that global solution $(a,m)$ satisfies 
\begin{align*}
&\||\N|^{s_1} a(t)\|_{\wh{L}^1}=O(t^{-\frac{s_1}{2}}) \;(t \to \infty) 
\quad \text{for all } s_1>0, \\
&\||\N|^{s_2} m(t)\|_{\wh{L}^1}=O(t^{-\frac{s_2+1}{2}}) \;(t \to \infty) 
\quad \text{for all } s_2>-1. 
\end{align*}
\end{cor}

\subsection{The proof of Theorem \ref{thm;Asympt}}  
By the Duhamel formula \eqref{eqn;IE}, $a(t)$ is written by 
\begin{align*}
&\begin{aligned}
a(t)-&G_1(t) \int_{\R^d} a_0(y) dy-G_2(t) \int_{\R^d} \tilde{m}_0(y) dy
-G_2(t) \int_0^\infty \int_{\R^d} \left(\wt{I}_{P}(a)a^2\right) dy d\t \\
-&\wt{G}_2^{(j,k)}(t) \int_0^\infty \int_{\R^d} \left( \frac{m_j m_k}{1+a}+\wt{K}^{(j,k)}(a)\right) dy d\t
\end{aligned} \\
&=: \sum_{i=1}^2 L_i(t)+\sum_{j=1}^4 N_j(t), 
\end{align*}
where $L_i(t)$ $(i=1,2)$, $N_j(t)$ $(j=1,2,3,4)$ are given by  
\begin{align*}
&L_1(t)=G^{1,1}(t)*a_0-G_1(t) \int_{\R^d} a_0(y) dy, \\
&L_2(t)=G^{1,2}(t)*m_0-G_2(t) \int_{\R^d} \tilde{m}_0(y) dy, \\
&N_1(t)=\int_0^t G^{1,2}(t-\t)*\left(\N (\wt{I}_P(a) a^2)\right) d\t
-G_2(t) \int_0^\infty \int_{\R^d} \left(\wt{I}_{P}(a)a^2\right) dy d\t, \\
&N_2(t)=\int_0^t G^{1,2}(t-\t)*\div((I(a)-1)m \otimes m) d\t-\wt{G}_2^{(j,k)}(t) \int_0^\infty \int_{\R^d} \left( \frac{m_j m_k}{1+a}\right) dy d\t, \\
&N_3(t)=\int_0^t G^{1,2}(t-\t)*\div \left(\mathcal{K}(a)-\frac{\kappa}{2} \Delta a^2 \text{Id}\right) d\t-\wt{G}_2^{(j,k)}(t) \int_0^\infty \int_{\R^d} \wt{K}^{(j,k)}(a) dy d\t, \\
&N_4(t)=\int_0^t G^{1,2}(t-\t)*\left(-\L(I(a)m)+\frac{\kappa}{2}\N \Delta a^2\right)d\t. 
\end{align*}
Thanks to Proposition \ref{prop;AP1}, we already obtain that for all $s>-d/p'$, 
$\dsp \lim_{t \to \infty} t^{\frac{d}{2p'}+\frac{s}{2}} \||\N|^s L_i(t)\|_{\wh{L}^p}=0$ with $i=1,2$. 
We now claim that 
${\dsp \lim_{t \to \infty} t^{\frac{d}{2p'}+\frac{s}{2}} \||\N|^s N_j(t)\|_{\wh{L}^p}=0}$ with $j=1,2,3,4$. 

\noindent
\underline{\it The estimate for $N_1(t)$}:  Now, we split $\wh{N}_1(t)$ which means the Fourier side of $N_1(t)$ as 
\begin{align*}
\wh{N}_1(t)=& -\wh{G}_2(t) \int_{t/2}^\infty c_d \mathcal{F}[a^2 \wt{I}_P(a)](\t,0) d\t \\
&+\int_0^{t/2} \mathcal{G}^{1,2}(t-\t) c_d i\xi \left(\mathcal{F}[a^2 \wt{I}_P(a)](\t,\xi)-\mathcal{F}[a^2 \wt{I}_P(a)](\t,0)\right) d\t \\
&+\int_0^{t/2} c_d \left(\mathcal{G}^{1,2}(t-\t) i\xi-\wh{G}_2(t)\right) \mathcal{F}[a^2 \wt{I}_P(a)](\t,0) d\t \\
&+\int_{t/2}^t c_d \mathcal{G}^{1,2}(t-\t) i\xi \mathcal{F}[a^2 \wt{I}_P(a)](\t,\xi) d\t \\
=&:\wh{N}_{1,1}(t)+\wh{N}_{1,2}(t)+\wh{N}_{1,3}(t)+\wh{N}_{1,4}(t). 
\end{align*}

Since 
$|e^{\lam_\pm(\xi)t}| \les e^{-c_0 |\xi|^2 t}$ if $\nu^2 \neq 4 \kappa$ and 
$|t |\xi|^2 e^{\frac{\nu}{2}|\xi|^2 t}| \les e^{-c_0 |\xi|^2 t}$ if $\nu^2 = 4 \kappa$, 
 it follows from  H\"older's inequality, \eqref{est;L1hat}, 
 $\||\xi|^\al e^{-c_0 t|\xi|^2}\|_{L^{p'}_\xi} \les t^{-\frac{d}{2p'}-\frac{\al}{2}}$ for all $\al>-d/p'$, and 
Theorem \ref{thm;Lp-L1} that for all $t>2$, 
\begin{align*}
t^{\frac{d}{2p'}+\frac{s}{2}}\||\N|^s N_{1,1}(t)\|_{\wh{L}^p}
\les& t^{\frac{d}{2p'}+\frac{s}{2}} \||\xi|^s e^{-c_0 |\xi|^2 t}\|_{L^{p'}_\xi} 
\left|\int_{t/2}^\infty \mathcal{F}[a^2 \wt{I}_P(a)](\t,0) d\t\right| \\
\les& \int_{t/2}^\infty \|(a^2 \wt{I}_P(a))(\t)\|_{L^1} d\t
\les  \int_{t/2}^\infty \|a(\t)\|_{L^2}^2 d\t \\
\les& \int_{t/2}^\infty \|a(\t)\|_{\wh{L}^1} \|a(\t)\|_{\wh{L}^\infty} d\t  
\les \int_{t/2}^\infty \t^{-\frac{d}{2}} d\t 
\les t^{-\frac{d}{2}+1}. 
\end{align*}
Therefore, we see $\||\N|^s N_{1,1}(t)\|_{\wh{L}^p}=o(t^{-\frac{d}{2}(1-\frac{1}{p})-\frac{s}{2}})$ $(t \to \infty)$. 

Performing the change of variables $\xi \mapsto \frac{\eta}{\sqrt{t-\t}}$ and using 
$\||\xi|^\al e^{-c_0 |\xi|^2}\|_{L^{p'}_\xi} \les 1$ if $\al>-d/p'$, 
it follows that 
\begin{align*}
\||\N|^s N_{1,2}(t)\|_{\wh{L}^p}
\les& \int_0^{t/2} \left\||\xi|^{s} e^{-c_0 (t-\t) |\xi|^2}
\left(\mathcal{F}[a^2 \wt{I}_P(a)](\t,\xi)-\mathcal{F}[a^2 \wt{I}_P(a)](\t,0)\right)\right\|_{L^{p'}_\xi} d\t \\
\les& \int_0^{t/2} (t-\t)^{-\frac{d}{2p'}-\frac{s}{2}} \\
& \times  \left\| |\eta|^s e^{-c_0 |\eta|^2} \left(\mathcal{F}[a^2 \wt{I}_P(a)]\left(\t,\frac{\eta}{\sqrt{t-\t}}\right)-\mathcal{F}[a^2 \wt{I}_P(a)](\t,0)\right) \right\|_{L^{p'}_\eta} d\t \\
\les& t^{-\frac{d}{2p'}-\frac{s}{2}} \int_0^{t/2} \psi_t (\t) d\t, 
\end{align*}
where $\psi_t(\t):=
\||\eta|^s e^{-c_0 |\eta|^2} (\mathcal{F}[a^2 \wt{I}_P(a)](\t,\frac{\eta}{\sqrt{t-\t}})-\mathcal{F}[a^2 \wt{I}_P(a)](\t,0))\|_{L^{p'}_\eta}$.  To show 
$\||\N|^s N_{1,2}(t)\|_{\wh{L}^p}=o(t^{-\frac{d}{2}(1-\frac{1}{p})-\frac{s}{2}})$, 
we need to confirm that 
\begin{equation} \label{est;psi} 
\lim_{t \to \infty} \int_0^{t/2} \psi_t(\t) d\t=0. 
\end{equation}
First of all, for fixed $M>0$, we easily show that $\lim_{t \to \infty} \int_0^M \psi_t (\t) d\t=0$. 
Actually, it follows from the definition of $\psi_t(\t)$ that 
\begin{align*}
  \int_0^M \psi_t (\t) d\t
 \les &\int_0^M \left\| |\eta|^s e^{-c_0 |\eta|^2} \int_{\R^d} (1-e^{-i y \cdot \frac{\eta}{\sqrt{t-\t}}}) a^2(\t,y) \wt{I}_P(a)(\t,y) dy\right\|_{L^{p'}_\eta} d\t. 
 \end{align*}
 By using the dominated convergence theorem, we obtain $\lim_{t \to \infty} \int_0^M \psi_t (\t) d\t=0$ as $t \to \infty$ 
 because it follows from Plancherel's theorem and  $a \in L^\infty(\R_+;\wh{L}^1) \cap L^2(\R_+;\fB{_{p,1}^{d/p}})$ 
 that 
 \begin{align*}
 &\left\| |\eta|^s e^{-c_0 |\eta|^2} \int_{\R^d} (1-e^{-i y \cdot \frac{\eta}{\sqrt{t-\t}}}) a^2(\t,y) \wt{I}_P(a)(\t,y) dy\right\|_{L^{p'}_\eta} \\
 &\les \left\| \int_{\R^d} (1-e^{-i y \cdot \frac{\eta}{\sqrt{t-\t}}}) a^2(\t,y) \wt{I}_P(a)(\t,y) dy\right\|_{L^\infty_\eta} \\
 &\les \|(a \wt{I}_P(a))(\t)\|_{\wh{L}^\infty} \|a(\t)\|_{\wh{L}^1} 
 \les \|a(\t)\|_{\wh{L}^\infty} \|a(\t)\|_{\wh{L}^1} \in L^1(0,M). 
 \end{align*} 

and $\lim_{t \to \infty} \psi_t(\t)=0$ a.e. $\t \in (0,M)$. On the other hand, for all $\ve>0$, 
we can choose $M>0$ such that $\int_M^\infty \psi_t(\t) d\t < \ve$ because it follows from Hausdorff-Young's inequality 
and Theorem \ref{thm;Lp-L1} that 
\begin{align*}
\int_M^\infty \psi_t(\t) d\t
\les \int_M^\infty \|(a \wt{I}_P(a))(\t)\|_{\wh{L}^\infty} \|a(\t)\|_{\wh{L}^1} d\t  
\les \int_M^\infty \t^{-\frac{d}{2}} d\t \les M^{1-\frac{d}{2}}. 
\end{align*}
Due to $d \ge 3$, it holds that $M^{1-\frac{d}{2}} < \ve$ if $M \gg 1$. 
Gathering the above arguments, we thus obtain that for all $\ve>0$, there exists some $M>0$ such that 
\begin{align*}
\int_0^{t/2} \psi_t(\t) d\t \les \int_0^M \psi_t(\t) d\t+\ve. 
\end{align*}
Therefore, we could confirm that \eqref{est;psi} holds true. And also, we can obtain 
$\||\N|^s N_{1,2}\|_{\wh{L}^p}=o(t^{-\frac{d}{2}(1-\frac{1}{p})-\frac{s}{2}})$ $(t \to \infty)$. 

As for the estimate on $N_{1,3}(t)$, we first consider the case of $\nu^2 \neq 4 \kappa$. 
Applying the mean value theorem to  $\wh{G}_\pm(t-\t)-\wh{G}_\pm(t)$ in $\wh{N}_{1,3}(t)$, we see that 
\begin{align*}
\||\N|^s N_{1,3}(t)\|_{\wh{L}^p} 
\les& \sum_{\s=\pm} \int_0^{t/2} 
\left\||\xi|^s \left(\int_0^1 \frac{\t \lam_\s(\xi)}{\sqrt{\nu^2-4 \kappa}} e^{(t-\theta \t)\lam_\s(\xi)} d\theta\right) 
\mathcal{F}[a^2 \wt{I}_P(a)](\t,0)\right\|_{L^{p'}_\xi} d\t \\
\les& \int_0^{t/2} \int_0^1 \|\t |\xi|^{s+2} e^{-c_0 (t-\theta \t)|\xi|^2} \mathcal{F}[a^2 \wt{I}_P(a)](\t,0)\|_{L^{p'}_\xi} d\theta d\t \\
\les& 
\int_0^{t/2} \int_0^1 \t (t-\theta \t)^{-\frac{d}{2p'}-\frac{s+2}{2}} \|\mathcal{F}[a^2 \wt{I}_P(a)](\t,0)\|_{L^{\infty}_\xi} d\theta d\t \\
\les& t^{-\frac{d}{2p'}-\frac{s}{2}-1} \int_0^{t/2} \t \|a(\t)\|_{L^2}^2 d\t. 
\end{align*}
On the other hand, in the case of $\nu^2 = 4 \kappa$, we obtain by applying 
the mean value theorem to $(t-\t)e^{-\frac{\nu}{2}|\xi|^2(t-\t)}-te^{-\frac{\nu}{2}|\xi|^2t}$, we also see that 
\begin{align*}
&\||\N|^s N_{1,3}(t)\|_{\wh{L}^p} \\
&\les \int_0^{t/2} 
\left\||\xi|^{s} \left\{\int_0^1 \t |\xi|^2 \left(1-\frac{\nu}{2}|\xi|^2 (t-\theta \t) \right) e^{-\frac{\nu}{2}(t-\theta \t)|\xi|^2} d\theta\right\}
\mathcal{F}[a^2 \wt{I}_P(a)](\t,0)\right\|_{L^{p'}_\xi} d\t \\
&\les \int_0^{t/2} \int_0^1 \|\t |\xi|^{s+2} e^{-c_0 (t-\theta \t)|\xi|^2} \mathcal{F}[a^2 \wt{I}_P(a)](\t,0)\|_{L^{p'}_\xi} d\theta d\t \\
&\les 
\int_0^{t/2} \int_0^1 \t (t-\theta \t)^{-\frac{d}{2p'}-\frac{s+2}{2}} \|\mathcal{F}[a^2 \wt{I}_P(a)](\t,0)\|_{L^{\infty}_\xi} d\theta d\t \\
&\les t^{-\frac{d}{2p'}-\frac{s}{2}-1} \int_0^{t/2} \t \|a(\t)\|_{L^2}^2 d\t. 
\end{align*}

Since $a \in L^\infty(\R_+;\wh{L}^1) \cap L^2(\R_+;\fB{_{p,1}^\frac{d}{p}})$, it follows from Hausdorff-Young's inequality and 
Theorem \ref{thm;Lp-L1} that for all $t>2$, 
\begin{align*}
t^{-1} \int_0^{t/2} \t \|a(\t)\|_{L^2}^2 d\t
\les& t^{-1} \int_0^{1} \t \|a(\t)\|_{L^2}^2 d\t+t^{-1} \int_1^{t/2} \t \|a(\t)\|_{L^2}^2 d\t \\
\les& t^{-1} \|a\|_{L^\infty(0,1;\wh{L}^1)} \|a\|_{L^2(0,1;\wh{L}^\infty)} 
+t^{-1} \int_1^t \t^{1-\frac{d}{2}} d\t. 
\end{align*}
As for the last integral term, it is straightforward to obtain 
\begin{equation*}
t^{-1} \int_1^t \t^{1-\frac{d}{2}} d\t 
\le 
\left\{
\begin{aligned}
&2 t^{-\frac{1}{2}}, &d=3, \\
&t^{-1} \log t, &d=4, \\
&\frac{2}{d-4} t^{-1}, &d \ge 5, 
\end{aligned}
\right. 
\end{equation*}
and thus, we can obtain that for  all $t>2$, 
$$
\lim_{t \to \infty} t^{-1} \int_0^{t/2} \t \|a(\t)\|_{L^2}^2 d\t=0. 
$$
Therefore, we arrive at $\||\N|^s N_{1,3}\|_{\wh{L}^p}=o(t^{-\frac{d}{2}(1-\frac{1}{p})-\frac{s}{2}})$ $(t \to \infty)$. 

Let us consider the estimate of $N_{1,4}(t)$. First of all, we obtain by the boundedness of $\mathcal{G}^{1,2}(t-\t) i \xi$ 
and Lemma \ref{lem;equi} that for all $s>-d/p$ and $1 \le p \le 2$, 
\begin{align*}
\||\N|^s N_{1,4}(t)\|_{\wh{L}^p}
\les& \int_{t/2}^t \||\xi|^s \mathcal{G}^{1,2}(t-\t) i\xi \F[a^2 \wt{I}_P(a)](\t,\xi)\|_{L^{p'}_\xi} d\t \\
\les& \int_{t/2}^t \|(a^2 \wt{I}_P(a))(\t)\|_{\wh{\dot{B}}{^s_{p,p'}}} d\t. 
\end{align*}

In the case of $-d/p'<s<d/p$, we obtain by using Lemma \ref{lem;A-P} that 
\begin{align*}
\int_{t/2}^t \|(a^2 \wt{I}_P(a))(\t)\|_{\wh{\dot{B}}{^s_{p,p'}}} d\t
\les& \int_{t/2}^t \|(a \wt{I}_P(a))(\t)\|_{\fB{_{p,1}^\frac{d}{p}}} \|a(\t)\|_{\fB{_{p,p'}^s}} d\t \\
\les& \int_{t/2}^t  \|a(\t)\|_{\fB{_{p,1}^\frac{d}{p}}} \|a(\t)\|_{\fB{_{p,p'}^s}} d\t \\
\les& \int_{t/2}^t  \t^{-\frac{d}{2}} \t^{-\frac{d}{2p'}-\frac{s}{2}} d\t
\les  t^{-\frac{d}{2p'}-\frac{s}{2}} t^{1-\frac{d}{2}}, 
\end{align*}
and thus, we have $\||\N|^s N_{1,4}(t)\|_{\wh{L}^p}=o(t^{-\frac{d}{2}(1-\frac{1}{p})-\frac{s}{2}})$ $(t \to \infty)$ because of 
$d \ge 3$. 

On the other hand, in the case of $s \ge d/p>0$, it follows from Lemma \ref{lem;bil} that 
\begin{align*}
\int_{t/2}^t \|(a^2 \wt{I}_P(a))(\t)\|_{\wh{\dot{B}}{^s_{p,p'}}} d\t
\les \int_{t/2}^t \|(a^2)(\t)\|_{\wh{\dot{B}}{^s_{p,p'}}} d\t 
\les \int_{t/2}^t  \|a(\t)\|_{\fB{_{p,1}^\frac{d}{p}}} \|a(\t)\|_{\fB{_{p,p'}^s}} d\t. 
\end{align*}
In a similar argument as in the case of $-d/p'<s<d/p$, we thus obtain 
$\||\N|^s N_{1,4}(t)\|_{\wh{L}^p}=o(t^{-\frac{d}{2}(1-\frac{1}{p})-\frac{s}{2}})$ $(t \to \infty)$. 
Gathering the above arguments, we complete the estimate of $N_1(t)$. 

\noindent
\underline{\it The estimate for $N_2(t)$, $N_3(t)$}: 
Mimicking the same procedure for the estimate of $N_1(t)$, 
the estimates of $N_2(t)$ and $N_3(t)$ can be obtained. 
We leave the details to the reader. 

\noindent
\underline{\it The estimate for $N_4(t)$}: 
Let us consider the estimate of $N_{4}(t)$. In this paper, we only consider the estimate for $\N \Delta a^2$. 
It follows from the same arguments as in the proof of Theorem \ref{thm;Lp-L1} and Lemma \ref{lem;equi} that 
\begin{align*}
\left\||\N|^s \int_0^t G^{1,2}(t-\t)*\left(\frac{\kappa}{2} \N \Delta a^2\right)(\t) d\t \right\|_{\wh{L}^p} 
\les& \int_0^t \||\xi|^s |\xi|^{-1} e^{-(t-\t)c_0|\xi|^2} \mathcal{F}[\N \Delta a^2](\t,\xi)\|_{L^{p'}_\xi} d\t \\
\les& \int_0^{t/2} (t-\t)^{-\frac{d}{2p'}-\frac{s+2}{2}} \|(a^2)(\t)\|_{\fB{_{p,\infty}^{-\frac{d}{p'}}}} d\t \\
&+ \int_{t/2}^t \|(\Delta a^2)(\t)\|_{\fB{_{p,p'}^s}} d\t. 
\end{align*}
Thanks to the product estimate obtained by \cite[Lemma 5.2]{Na2}, we see that 
\begin{align*}
 \int_0^{t/2} (t-\t)^{-\frac{d}{2p'}-\frac{s+2}{2}} \|(a^2)(\t)\|_{\fB{_{p,\infty}^{-\frac{d}{p'}}}} d\t 
 \les& t^{-\frac{d}{2p'}-\frac{s+2}{2}} \int_0^{t/2} \|a(\t)\|_{\fB{_{p,\infty}^{-\frac{d}{p'}}}} \|a(\t)\|_{\fB{_{p,1}^\frac{d}{p}}} d\t \\
 \les& t^{-\frac{d}{2p'}-\frac{s+2}{2}} \int_0^{t/2} \langle \t\rangle^{-\frac{d}{2}} d\t 
 \les t^{-\frac{d}{2p'}-\frac{s+2}{2}}. 
\end{align*}

On the other hand, noting that  $\Delta a^2=2|\N a|^2+2 a \Delta a$, 
we obtain from the same arguments as obtaining the estimate for $N_{1,4}(t)$ that 
\begin{align*}
\int_{t/2}^t \|(\Delta a^2)(\t)\|_{\fB{_{p,p'}^s}} d\t
\les \int_{t/2}^t \t^{-\frac{d}{2p'}-\frac{s+d+2}{2}} d\t 
\les t^{-\frac{d}{2p'}-\frac{s+d}{2}}. 
\end{align*}

Analogously, one can obtain the estimate for $\mathcal{L}(I(a)m)$. 
Therefore, we arrive at $\||\N|^s N_4(t)\|_{\wh{L}^p}=o(t^{-\frac{d}{2}(1-\frac{1}{p})-\frac{s}{2}})$ $(t \to \infty)$. 

Combining the above estimates, we obtain the desired estimate \eqref{est;AE}. \qed 

\begin{rem}
By performing the similar calculations, we are able to obtain \eqref{est;AE2} and \eqref{est;AE3}. 
The detals are left to the readers. 
\end{rem}

\vskip2mm
\noindent
{\bf Acknowledgments}. 
The first author is supported by Grant-in-Aid for Scientific Research (C) JP22K03374. 
The second author is supported by JSPS Grant-in-Aid for Early-Career Scientists JP22K13936.

\sect{Appendix: The product estimates} \label{sect;appendix}

In this section, we give the proof of product estimates and bilinear estimates as used in \S \ref{sect;GWP}-\S \ref{sect;AP}. 

\begin{lem}[{\it Lemma 2.6 in} \cite{Na}] \label{lem;offd1}
Let $s_1$,$s_2 \in \R$, $1 \le p,\s \le \infty$, $T \in \R_+$ and 
$1 \le r,r_1,r_2 \le \infty$ satisfy 
$\frac{1}{r}=\frac{1}{r_1}+\frac{1}{r_2}$.  
If $s_1 < \frac{d}{p}$, then 
there exists some constant $C>0$ such that 
\begin{align} \label{est;offd2}
\left\|\sum_{k \in \Z}S_{k-1}f g_k \right\|_{\wt{L^{r}_T(}\fB{_{p,\s}^{s_1+s_2-\frac{d}{p}}})}
\le C \|f\|_{\wt{L^{r_1}_T(}\fB{_{p,\s}^{s_1}})}
      \|g\|_{\wt{L^{r_2}_T(}\fB{_{p,\infty}^{s_2}})}. 
\end{align}
\end{lem}

\begin{lem}[{\it Lemma 2.5 in} \cite{Na}] \label{lem;diagonal1}
Let $s_i \in \R$, $1 \le p,p_i,r,r_i,\s \le \infty$ $(i=1,2)$ satisfying     
$$
  \frac{1}{p} \le \frac{1}{p_1} + \frac{1}{p_2}, \quad 
  \frac{1}{r}=\frac{1}{r_1}+\frac{1}{r_2}, \quad s_1+s_2+d\min \left(0,1-\frac{1}{p_1}-\frac{1}{p_2}\right)>0. 
$$
There exists some constant $C>0$ such that 
\begin{equation} \label{est;prod_B}
\Bigg\|\sum_{k\in\Z} f_k \tilde{g}_k\Bigg\|_{\wt{L^r_T(}\fB{_{p,\s}^{s_1+s_2+\frac{d}{p}-\frac{d}{p_1}-\frac{d}{p_2}}})}
\le C\|f\|_{\wt{L^{r_1}_T(}\fB{_{p_1,\s}^{s_1}})}
     \|g\|_{\wt{L^{r_2}_T(}\fB{_{p_2,\infty}^{s_2}})}. 
\end{equation}
\end{lem}
We here omit the proof of Lemmas \ref{lem;offd1} and \ref{lem;diagonal1} because 
these are same as the proof of following Lemmas \ref{lem;offd2} and \ref{lem;diagonal2}, respectively. 

\begin{lem} \label{lem;offd2}
Let $s_1$,$s_2 \in \R$, $1 \le p,\s \le \infty$, $T \in \R_+$ and 
$1 \le r,r_1,r_2 \le \infty$ satisfy 
$\frac{1}{r}=\frac{1}{r_1}+\frac{1}{r_2}$.  
If $s_1 < \frac{d}{p}$, then 
there exists some constant $C>0$ such that 
\begin{align} \label{est;offd2}
\left\|e^{\sqrt{c_0 t}|\N|}\sum_{k \in \Z}S_{k-1}f g_k \right\|_{\wt{L^{r}_T(}\fB{_{p,\s}^{s_1+s_2-\frac{d}{p}}})}
\le C \|F\|_{\wt{L^{r_1}_T(}\fB{_{p,\s}^{s_1}})}
      \|G\|_{\wt{L^{r_2}_T(}\fB{_{p,\infty}^{s_2}})}. 
\end{align}
Here we recall that $F:=e^{\sqrt{c_0 t}|\N|}f$, $G:=e^{\sqrt{c_0 t}|\N|}g$. 
\end{lem}

\begin{proof}[The proof of Lemma \ref{lem;offd2}] 
The strategy of the proof is inspired by Song-Xu \cite{So-Xu}. 
By the definition of $\B_t$, we firstly see that  
\begin{align*}
e^{\sqrt{c_0 t}|\N|}\sum_{k \in \Z}S_{k-1}f g_k
= \sum_{k \in \Z} \B_t(S_{k-1}F, G_k). 
\end{align*}
Here we recall that $\B_t(f,g):=e^{\sqrt{c_0 t}|\N|}(e^{-\sqrt{c_0 t}|\N|}f e^{-\sqrt{c_0 t}|\N|}g)(x)$ 
with the constant $c_0>0$. 
Since $\|\wh{\phi}_j\|_{L^\infty_\xi} \le 1$, we obtain by Lemma \ref{lem;B_t} and H\"older's inequality 
that 
\begin{equation} \label{est;sss}
\begin{aligned}
\left\|\Dj \sum_{k \in \Z} \B_t(S_{k-1}F, G_k)\right\|_{L^r(I;\wh{L}^p)} 
\les& \sum_{|j-k| \le 4} \|\wh{\phi}_j\|_{L^\infty_\xi} \|\B_t(S_{k-1}F, G_k)\|_{L^r(I;\wh{L}^p)} \\
\les& \sum_{|j-k| \le 4} \|S_{k-1}F\|_{L^{r_1}(I;\wh{L}^\infty)} \|G_k\|_{L^{r_2}(I;\wh{L}^p)}. 
\end{aligned}
\end{equation}
Since it follows from  Lemma \ref{lem;Bern} that 
$$
\|S_{k-1}F\|_{L^{r_1}(I;\wh{L}^\infty)} 
\les \sum_{\ell \le k-2}\|F_\ell\|_{L^{r_1}(I;\wh{L}^\infty)} 
\les \sum_{\ell \le k-2} 2^{\frac{d}{p}\ell} \|F_\ell\|_{L^{r_1}(I;\wh{L}^p)}, 
$$
we obtain by multiplying the both sides $2^{(s_1+s_2-d/p)j}$ and combining the above estimate with \eqref{est;sss} that 
\begin{align*}
&2^{(s_1+s_2-\frac{d}{p})j}\left\|\Dj \sum_{k \in \Z} \B_t(S_{k-1}F, G_k)\right\|_{L^r(I;\wh{L}^p)} \\
&\les 2^{(s_1+s_2-\frac{d}{p})j} \sum_{|j-k| \le 4} 
\left(\sum_{\ell \le k-2} 2^{(\frac{d}{p}-s_1) \ell} 2^{s_1\ell} \|F_\ell\|_{L^{r_1}(I;\wh{L}^p)}\right)
\|G_k\|_{L^{r_2}(I;\wh{L}^p)} \\
&\les \|G\|_{\wt{L^{r_2}_T(}\fB{_{p,\infty}^{s_2}})}
\sum_{|j-k| \le 4} 2^{(s_1+s_2-\frac{d}{p})(j-k)}
\sum_{\ell \le k-2} 2^{(s_1-\frac{d}{p})(k-\ell)}  2^{s_1 \ell} 
\|F_\ell\|_{L^{r_1}(I;\wh{L}^p)} 
\end{align*}
Noting that $s_1<d/p$ and taking the $\ell^\s(\Z)$-norm of the above estimate, we obtain the desired estimate \eqref{est;offd2} 
because of $\sum_{m \ge 2} 2^{(s_1-d/p)m} \les 1$. 
\end{proof}

\begin{lem} \label{lem;diagonal2}
Let $s_i \in \R$, $1 \le p,p_i,r,r_i,\s \le \infty$ $(i=1,2)$ satisfying     
$$
  \frac{1}{p} \le \frac{1}{p_1} + \frac{1}{p_2}, \quad 
  \frac{1}{r}=\frac{1}{r_1}+\frac{1}{r_2}, \quad s_1+s_2+d\min \left(0,1-\frac{1}{p_1}-\frac{1}{p_2}\right)>0. 
$$
There exists some constant $C>0$ such that 
\begin{equation} \label{est;prod_B}
\Bigg\|e^{\sqrt{c_0 t}|\N|} \sum_{k\in\Z} f_k \tilde{g}_k\Bigg\|_{\wt{L^r_T(}\fB{_{p,\s}^{s_1+s_2+\frac{d}{p}-\frac{d}{p_1}-\frac{d}{p_2}}})}
\le C\|F\|_{\wt{L^{r_1}_T(}\fB{_{p_1,\s}^{s_1}})}
     \|G\|_{\wt{L^{r_2}_T(}\fB{_{p_2,\infty}^{s_2}})}. 
\end{equation}
\end{lem}

\begin{proof}[The outlined proof of Lemma \ref{lem;diagonal2}] 
In the case of $\frac{1}{p_1} + \frac{1}{p_2} \le 1$,   
we take $1 \le q \le \infty$ such that $\frac{1}{q} = \frac{1}{p_1} + \frac{1}{p_2}$.    
Since $\wh{\phi}_j(\wh{f}_k*\wh{\tilde{g}}_k)=0\;(k \le j-4)$, we have by using  
H\"older's inequality and Lemma \ref{lem;B_t} that 
\begin{equation*}
\begin{aligned}
\left\|\Dj e^{\sqrt{c_0 t}|\N|} \sum_{k\in\Z} f_k \tilde{g}_k\right\|_{L^r(I;\wh{L}^{p})}
\les& \sum_{j-k \le 3} \|\wh{\phi}_j\|_{L^{\gm'}} \|\B_t(F_k, \tilde{G}_k)\|_{L^r(I;\wh{L}^q)} \\
\les& 2^{-d(\frac{1}{p}-\frac{1}{p_1}-\frac{1}{p_2})j}
    \sum_{j-k\le 3} \|F_k\|_{L^{r_1}(I;\wh{L}^{p_1})} \|\tilde{G}_k\|_{L^{r_2}(I;\wh{L}^{p_2})},    
\end{aligned}
\end{equation*}
where $\gm$ satisfies $\frac{1}{p}= \frac{1}{\gm}+\frac{1}{q}-1$ and we have used 
$$
e^{\sqrt{c_0 t}|\N|}\sum_{k \in \Z}f_k \tilde{g}_k
= \sum_{k \in \Z} \B_t(F_k, \tilde{G}_k). 
$$ 
Analogously, in the case of $\frac{1}{p_1}+\frac{1}{p_2} > 1$, we have by noting $p_1<p_2'$ and using 
Lemma \ref{lem;B_t} that 
\begin{equation*} 
\begin{aligned}
\left\|\Dj e^{\sqrt{c_0 t}|\N|} \sum_{k\in\Z} f_k \tilde{g}_k\right\|_{L^r(I;\wh{L}^{p})}
\les& \sum_{j-k \le 3} \|\wh{\phi}_j\|_{L^{p'}} \|\B_t(F_k, \tilde{G}_k)\|_{L^r(I;\wh{L}^1)} \\ 
\les& 2^{\frac{d}{p'}j} 
      \sum_{j-k \le 3} \|F_k\|_{L^{r_1}(I;\wh{L}^{p_1})}\|\tilde{G}_k\|_{L^{r_2}(I;\wh{L}^{p_2})}. 
\end{aligned}
\end{equation*}
Because the remaining part is completely same as the proof of Lemma 2.5 in \cite{Na}, 
we leave the details to the reader. 
\end{proof}

\begin{lem} \label{lem;diagonal3}
Let $s_i \in \R$, $1 \le p,p_i,r,r_i \le \infty$ $(i=1,2)$ satisfy 
$$
\frac{1}{p}\le\frac{1}{p_1}+\frac{1}{p_2}, \quad
s_1+s_2+d\min\left(0,1-\frac{1}{p_1}-\frac{1}{p_2}\right) \ge 0, \quad 
\frac{1}{r}=\frac{1}{r_1}+\frac{1}{r_2}. 
$$
Then there exists some constant $C>0$ such that 
\begin{align} \label{est;prod_diag}
\left\|e^{\sqrt{c_0 t}|\N|}\sum_{k \in\Z}f_k\tilde{g}_k\right\|_{
\wt{L_T^r(}\fB{_{p,\infty}^{s_1+s_2+\frac{d}{p}-\frac{d}{p_1}-\frac{d}{p_2}}})}
\le C\|F\|_{\wt{L^{r_1}_T(}\fB{_{p_1,1}^{s_1}})}\|G\|_{\wt{L^{r_2}_T(}\fB{_{p_2,\infty}^{s_2}})}. 
\end{align}
\end{lem}

\begin{lem}  \label{lem;offd3} 
Let $s_i \in \R$, $1 \le p,p_i,r,r_i\lam \le \infty$ $(i=1,2)$ satisfying 
$$
  \frac{1}{p}\le\frac{1}{p_1}+\frac{1}{p_2}, 
  \quad \frac{1}{p} \le \frac{1}{p_2}+\frac{1}{\lam} \le 1, \quad p_1 \le \lam, \quad 
  \frac{1}{r}=\frac{1}{r_1}+\frac{1}{r_2}. 
$$ 
Then there exists some constant $C>0$ such that the following estimates hold:
\begin{equation} \label{est;prod_C1}
\begin{aligned}
\left\|e^{\sqrt{c_0 t}|\N|}
\sum_{k\in\Z} S_{k-1}f g_k\right\|_{\wt{L^r_T(}\fB{_{p,\s}^{s_1+s_2+\frac{d}{p}-\frac{d}{p_1}-\frac{d}{p_2}}})}
\le C\|G\|_{\wt{L^{r_2}_T(}\fB{_{p_2,\s}^{s_2}})} 
\begin{cases}
\|F\|_{\wt{L^{r_1}_T(}\fB{_{p_1,\infty}^{s_1}})} 
   &if\; s_1+\frac{n}{\lam}<\frac{d}{p_1}, \\
\|F\|_{\wt{L^{r_1}_T(}\fB{_{p_1,1}^{s_1}})} 
   &if\; s_1+\frac{n}{\lam}=\frac{d}{p_1},   
\end{cases}
\end{aligned}
\end{equation}
\begin{equation} \label{est;prod_C2}
\begin{aligned}
\left\|e^{\sqrt{c_0 t}|\N|} \sum_{k \in \Z}S_{k-1}f g_k\right\|_{\wt{L^r_T(}\fB{_{p,\s}^{s_1+s_2+\frac{d}{p}-\frac{d}{p_1}-\frac{d}{p_2}}})}
\le C   
\begin{cases}
\|F\|_{\wt{L^{r_1}_T(}\fB{_{p_1,\s}^{s_1}})}
\|G\|_{\wt{L^{r_2}_T(}\fB{_{p_2,\infty}^{s_2}})}
         &if\; s_1+\frac{d}{\lam}<\frac{d}{p_1}, \\
\|F\|_{\wt{L^{r_1}_T(}\fB{_{p_1,1}^{s_1}})}
\|G\|_{\wt{L^{r_2}_T(}\fB{_{p_2,\s}^{s_2}})}
         &if\; s_1+\frac{d}{\lam}=\frac{d}{p_1}. 
\end{cases}
\end{aligned}
\end{equation}
\end{lem}
The proof of Lemmas \ref{lem;diagonal3} and \ref{lem;offd3} is same as the one of Lemma 7.2 in \cite{Na2}. 
We here omit they. 

\begin{proof}[The proof of Lemmas \ref{lem;bil2} and \ref{lem;bil_anal}] 
Since $-2+d/p<d/p$, it follows from Lemma \ref{lem;offd2} with $(s_1,s_2)=(-2+d/p,d/p)$ that 
\begin{align*}
\left\|e^{\sqrt{c_0 t}|\N|}\sum_{k \in \Z}S_{k-1}f g_k \right\|_{\wt{L^{r}_T(}\fB{_{p,\s}^{-2+\frac{d}{p}}})}
\les \|F\|_{\wt{L^{r_1}_T(}\fB{_{p,\s}^{-2+\frac{d}{p}}})}
      \|G\|_{\wt{L^{r_2}_T(}\fB{_{p,\infty}^{\frac{d}{p}}})}. 
\end{align*}      
Analogously, we easily see that 
\begin{align*}
\left\|e^{\sqrt{c_0 t}|\N|}\sum_{k \in \Z}f_k S_{k-1} g\right\|_{\wt{L^{r}_T(}\fB{_{p,\s}^{-2+\frac{d}{p}}})}
\les \|F\|_{\wt{L^{r_3}_T(}\fB{_{p,\infty}^{\frac{d}{p}}})}
      \|G\|_{\wt{L^{r_4}_T(}\fB{_{p,\s}^{-2+\frac{d}{p}}})}. 
\end{align*}
As for the diagonal part, we obtain by using Lemma \ref{lem;diagonal2} with $p=p_1=p_2$ that for all $1 \le p <d$, 
\begin{equation*}
\Bigg\|e^{\sqrt{c_0 t}|\N|} \sum_{k\in\Z} f_k \tilde{g}_k\Bigg\|_{\wt{L^r_T(}\fB{_{p,\s}^{-2+\frac{d}{p}}})}
\le C\|F\|_{\wt{L^{r_1}_T(}\fB{_{p,\s}^{-2+\frac{d}{p}}})}
     \|G\|_{\wt{L^{r_2}_T(}\fB{_{p,\infty}^{\frac{d}{p}}})} 
\end{equation*}
because of $2(-1+d/p)+\min(0,1-2/p)>0$ if $1 \le p < d$.   
By similar arguments, the assertion of Lemma \eqref{lem;bil2} follows from Lemmas \ref{lem;offd1} and \ref{lem;diagonal1}. 
\end{proof} 

\begin{proof}[The proof of Lemma \ref{lem;bil_anal2}] 
By Bony's para-product formula (cf. \cite{B}), $ e^{\sqrt{c_0 t}|\N|} fg$ is broken down 
\begin{align*}
  e^{\sqrt{c_0 t}|\N|}fg 
    =& e^{\sqrt{c_0 t}|\N|}
        \left( \sum_{k \in \mathbb{Z}}S_{k-1}f g_k
      + \sum_{k \in \mathbb{Z}}f_k S_{k-1} g
      + \sum_{k \in \mathbb{Z}}f_k\tilde{g}_k\right) \\ 
    =& \sum_{k \in \Z} \B_t(S_{k-1}F, G_k) 
    +\sum_{k \in \Z} \B_t(F_k, S_{k-1} G)
    +\sum_{k \in \Z} \B_t(F_k, \tilde{G}_k). 
\end{align*}
Noting that $\wh{\phi}_j (\wh{S_{k-1}}f * \wh{g}_k)=0$ for $k \in \Z$ with $|j-k| \ge 5$, 
it follows from Lemma \ref{lem;B_t} and H\"older's inequality that  
\begin{equation} \label{est;bil1}
\begin{aligned}
\left\|\Dj \sum_{k \in \Z} \B_t(S_{k-1}F, G_k)\right\|_{L^r(I;\wh{L}^p)} 
  \les& \sum_{|j-k| \le 4}
         \|\wh{\phi}_j\|_{L^\infty_\xi} \|\B_t(S_{k-1}F, G_k)\|_{L^r(I;\wh{L}^p)} \\
  \les& \sum_{|j-k| \le 4}
         \|S_{k-1}F\|_{L^{r_3}(I;\wh{L}^\infty)} 
         \|G_k\|_{L^{r_4}(I;\wh{L}^p)} \\
  \les& \|F\|_{L^{r_3}(I;\wh{L}^\infty)} 
        \sum_{|j-k| \le 4}\|G_k\|_{L^{r_4}(I;\wh{L}^p)}. 
\end{aligned}
\end{equation}
Similarly as in the estimate \eqref{est;bil1}, we have by 
using Lemma \ref{lem;B_t} and H\"older's inequality that  
\begin{equation} \label{est;bil2}
\left\|\Dj \sum_{k \in \Z} \B_t(F_k, S_{k-1}G)\right\|_{L^r(I;\wh{L}^p)} 
  \les \|G\|_{L^{r_2}(I;\wh{L}^\infty)} 
        \sum_{|j-k| \le 4}\|F_k\|_{L^{r_1}(I;\wh{L}^p)}. 
\end{equation}
On the other hand, noting that $\wh{\phi}_j (\wh{f}_k*\wh{\tilde{g}}_k) =0\;(k \le j-4)$, 
we have by Lemma \ref{lem;B_t} that  
\begin{equation}  \label{est;bil3}
 \begin{aligned}
   \left\|\Dj \sum_{k\in\Z} \B_t(F_k, \tilde{G}_k)\right\|_{L^r(I;\wh{L}^{p})}
   \les& \sum_{k \ge j-3}
        \|\B_t(F_k, \tilde{G}_k)\|_{L^r(I;\wh{L}^{p})} \\ 
   \les&  \sum_{k \ge j-3}
          \|F_k\|_{L^{r_1}(I;\wh{L}^p)} \|\tilde{G}_k\|_{L^{r_2}(I;\wh{L}^{\infty})} \\
   \les& \|G\|_{L^{r_2}(I;\wh{L}^{\infty})}
         \sum_{k \ge j-3} \|F_k\|_{L^{r_1}(I;\wh{L}^p)}. 
\end{aligned}
\end{equation}
Gathering \eqref{est;bil1}-\eqref{est;bil3}, we obtain 
\begin{align*}
  \left\|\Dj e^{\sqrt{c_0 t}|\N|}fg \right\|_{L^r(I;\wh{L}^p)}  
  \les& \|G\|_{L^{r_2}(I;\wh{L}^\infty)} 
        \sum_{|j-k| \le 4}\|F_k\|_{L^{r_1}(I;\wh{L}^p)} \\
      &+\|F\|_{L^{r_3}(I;\wh{L}^\infty)} 
        \sum_{|j-k| \le 4}\|G_k\|_{L^{r_4}(I;\wh{L}^p)}. 
\end{align*}
Multiplying the both sides by $2^{sj}$ and taking $\ell^\s(\Z)$-norm, 
we see that  
\begin{align*}
\|e^{\sqrt{c_0 t}|\N|}fg\|_{L^r(I;\fB{_{p,\s}^s})}
 \les \|G\|_{L^{r_2}(I;\wh{L}^\infty)} S_{1,\s}(F)
    + \|F\|_{L^{r_3}(I;\wh{L}^\infty)} S_{2,\s}(G),                                                    
\end{align*}
where $S_{1,\s}(F)$ and $S_{2,\s}(G)$ are defined as 
\begin{align*}
&S_{1,\s}(F):=\bigg\|\bigg\{ 2^{sj}\sum_{|j-k|\le 4} \|F_k\|_{L^{r_1}(I;\wh{L}^p)} \bigg\}_{j \in \Z}\bigg\|_{\ell^\s},  \\
&S_{2,\s}(G):=\bigg\|\bigg\{2^{sj}\sum_{k \ge j-4} \|G_k\|_{L^{r_4}(I;\wh{L}^p)} \bigg\}_{j \in \Z} \bigg\|_{\ell^\s}. 
\end{align*}
In the case of $1 \le \s<\infty$, we obtain by using Minkowski's inequality that 
\begin{align*}
  S_{2,\s}(G) \les& \sum_{k \ge j-4} \bigg(\sum_{j\in\Z} 2^{sj\s} \|G_k\|_{L^{r_4}(I;\wh{L}^{p})}^\s \bigg)^{1/\s} \\
       =& \sum_{l \le 4} 2^{sl} \bigg(\sum_{j \in \Z} 2^{s(j-l)\s}\|G_{j-l}\|_{L^{r_4}(I;\wh{L}^{p})}^\s \bigg)^{1/\s}
       \les \|G\|_{\wt{L^{r_4}(I};\fB{_{p,\s}^s})},  
\end{align*}
where $l:=j-k$. On the other hand, it holds that 
\begin{align*}
       S_{2,\infty}(G) 
       = \sum_{l \le 4} 2^{s(j-l)} 2^{sl}\|G_{j-l}\|_{L^{r_4}(I;\wh{L}^{p})}
       \les \|G\|_{\wt{L^{r_4}(I};\fB{_{p,\infty}^s})}. 
\end{align*}
Similarly, we also have by using Minkowski's inequality that for all $1 \le \s \le \infty$, 
\begin{align*}
  S_{1,\s}(F) \les \|F\|_{\wt{L^{r_1}(I};\fB{_{p,\s}^s})}. 
\end{align*} 
Gathering the above estimates, we complete the proof of Lemma \ref{lem;bil_anal}. 
\end{proof}

\begin{proof}[The proof of Lemma \ref{lem;bil5}] 
In a similar way to the proof of Lemma \ref{lem;bil_anal2}, 
it follows from $\wh{\phi}_j (\wh{S_{k-1}}f * \wh{g}_k)=0$ for $k \in \Z$ with $|j-k| \ge 5$ 
, Lemma \ref{lem;B_t} and H\"older's inequality that  
\begin{align} \label{est;bil4}
&\left\|\Dj \sum_{k \in \Z} \B_t(S_{k-1}F, G_k)\right\|_{L^r(I;\wh{L}^1)}  
\les \|F\|_{L^{r_3}(I;\wh{L}^\infty)} 
        \sum_{|j-k| \le 4}\|G_k\|_{L^{r_4}(I;\wh{L}^1)}, \\
&\left\|\Dj \sum_{k \in \Z} \B_t(F_k, S_{k-1}G)\right\|_{L^r(I;\wh{L}^1)} \label{est;bil5}
  \les \|G\|_{L^{r_2}(I;\wh{L}^\infty)} 
        \sum_{|j-k| \le 4}\|F_k\|_{L^{r_1}(I;\wh{L}^1)}. 
\end{align}
Notice that $\wh{\phi}_j (\wh{f}_k*\wh{\tilde{g}}_k) =0\;(k \le j-4)$, 
we have by Lemma \ref{lem;B_t} that  
\begin{equation}  \label{est;bil6}
 \begin{aligned}
   \left\|\Dj \sum_{k\in\Z} \B_t(F_k, \tilde{G}_k)\right\|_{L^r(I;\wh{L}^{1})}
   \les& \sum_{k \ge j-3} 
        \|\B_t(F_k, \tilde{G}_k)\|_{L^r(I;\wh{L}^{1})} \\ 
   \les&  \sum_{k \ge j-3}
          \|F_k\|_{L^{r_3}(I;\wh{L}^\infty)} \|\tilde{G}_k\|_{L^{r_4}(I;\wh{L}^{1})} \\
   \les& \|G\|_{\wt{L^{r_4}(I};\wh{\dot{B}}{_{1,\infty}^0})} 
         \sum_{k \ge j-3} \|F_k\|_{L^{r_3}(I;\wh{L}^1)}. 
\end{aligned}
\end{equation}
Gathering \eqref{est;bil4}-\eqref{est;bil6} and using Lemma \ref{lem;equi}, we obtain that for all $j \in \Z$, 
\begin{align*}
  \left\|\Dj e^{\sqrt{c_0 t}|\N|}fg \right\|_{L^r(I;\wh{L}^1)}  
  \les& \|F\|_{\wt{L^{r_1}(I;}\wh{\dot{B}}{_{1,\infty}^0})}  \|G\|_{L^{r_2}(I;\fB{_{\infty,1}^0})} \\
      &+\|F\|_{L^{r_3}(I;\fB{_{\infty,1}^0})} 
        \|G\|_{\wt{L^{r_4}(I};\wh{\dot{B}}{_{1,\infty}^0})} 
\end{align*}
and thus, we have the desired estimate \eqref{est;prod_bl} 
to take the supremum of the left hand side of the above estimate with respect to $j \in \Z$. 
\end{proof}

%


\end{document}